\documentclass{amsart}

\usepackage[colorlinks=true,linkcolor=blue,citecolor=blue]{hyperref}
\usepackage[margin=1.35in]{geometry}

\newtheorem{theorem}{Theorem}[section]
\newtheorem{prop}[theorem]{Proposition}
\newtheorem{lemma}[theorem]{Lemma}
\newtheorem{corollary}[theorem]{Corollary}

\theoremstyle{definition}
\newtheorem{definition}[theorem]{Definition}
\newtheorem{example}[theorem]{Example}
\newtheorem{ass}[theorem]{Assumption}

\theoremstyle{remark}
\newtheorem{remark}[theorem]{Remark}
\newtheorem{notation}[theorem]{Notation}

\numberwithin{equation}{section}

\def\A{{\bf A}}
\def\B{\mathcal{B}}
\def\C{{\bf C}}
\def\E{{E}}
\def\F{\mathcal{F}}
\def\K{{\bf K}}

\def\L{\mathcal{L}}
\def\P{{P}}
\def\R{\mathbb{R}}

\renewcommand{\P}{P}

\def\domain{{\bf D}}

\def\f{b}
\def\g{\sigma}

\def\lip{\kappa}

\def\x{x}

\def\xsops{x^\ast}
\def\period{p}
\def\orbit{{\bf O}}

\def\ve{\varepsilon}

\def\delay{\tau}
\def\exit{\rho}

\def\ball{{\bf B}}
\def\sphere{{\bf S}}

\newcommand{\be}{\begin{equation}}
\newcommand{\ee}{\end{equation}}

\newcommand{\lb}{\left(}
\newcommand{\rb}{\right)}
\newcommand{\lsb}{\left[}
\newcommand{\rsb}{\right]}
\newcommand{\lcb}{\left\{}
\newcommand{\rcb}{\right\}}

\newcommand{\norm}[1]{\lVert{#1}\rVert}

\newcommand{\dotx}{\dot{\x}}

\newcommand{\X}{X}

\begin{document}

\title[Small noise asymptotics for SDDEs]{Exit time asymptotics for small noise stochastic delay differential equations*}

\date{\today}

\thanks{*This research was supported in part by NSF grants DMS-1206772 and DMS-1148284.}

\author{David Lipshutz}
\address{Faculty of Electrical Engineering \\ Technion --- Israel Institute of Technology \\ Haifa, Israel}
\email{lipshutz@technion.ac.il}

\subjclass[2010]{60F10, 34K50, 60J25}

\keywords{Uniform large deviation principle, stochastic delay differential equation, small noise asymptotics, exit time}

\dedicatory{Technion --- Israel Institute of Technology}

\begin{abstract}
Dynamical system models with delayed dynamics and small noise arise in a variety of applications in science and engineering. In many applications, stable equilibrium or periodic behavior is critical to a well functioning system. Sufficient conditions for the stability of equilibrium points or periodic orbits of certain deterministic dynamical systems with delayed dynamics are known and it is of interest to understand the sample path behavior of such systems under the addition of small noise. We consider a small noise stochastic delay differential equation (SDDE) with coefficients that depend on the history of the process over a finite delay interval. We obtain asymptotic estimates, as the noise vanishes, on the time it takes a solution of the stochastic equation to exit a bounded domain that is attracted to a stable equilibrium point or periodic orbit of the corresponding deterministic equation. To obtain these asymptotics, we prove a sample path large deviation principle (LDP) for the SDDE that is uniform over initial conditions in bounded sets. The proof of the uniform sample path LDP uses a variational representation for exponential functionals of strong solutions of the SDDE. We anticipate that the overall approach may be useful in proving uniform sample path LDPs for a broad class of infinite-dimensional small noise stochastic equations. 
\end{abstract}

\maketitle

\newpage

\tableofcontents

\newpage

\section{Introduction}

\subsection{Overview}\label{sec:overview}

Dynamical system models with delayed signaling effects arise in a wide variety of applications in science and engineering. Examples include Internet congestion control models \cite{PAG02,PAP04a,PAP04b,PEE07}, neuronal models \cite{Coombes2009,Roxin2005,Roxin2011} and biochemical models of gene regulation \cite{Anderson2007,Bratsun2005,Mather2009}. In many applications, stable equilibrium behavior or periodic oscillatory behavior is critical to a well functioning system and there is a sizable literature on conditions for the stability of equilibrium points \cite{Driver1965,Hale1965,Huang1989,Razumikhin1956,Razumikhin1960} and periodic orbits \cite{Chow1988,Ivanov1999,Kaplan1975,Kaplan1977,Mallet-Paret2011,Mallet-Paret1996b,Mallet-Paret1996a,Stoffer2008,Stoffer2011,Walther1995,Walther2001b,Walther2001a,Walther2003,Wu2003,Xie1991,Xie1992,Xie1993} of deterministic delay differential equations (DDEs). Frequently, small noise is present and it is of interest to understand its effect on the dynamics, especially near stable equilibrium points and periodic orbits (of the corresponding deterministic system). While solutions of the deterministic system that start near a stable equilibrium point or periodic orbit will remain near the equilibrium point or periodic orbit for all time, solutions of the small noise stochastic system will eventually exit any bounded domain that contains the equilibrium point or periodic orbit (provided the noise coefficient is uniformly nondegenerate, see Assumption \ref{ass:non-degenerate} below). The main focus of this work is to estimate the time it takes solutions of the small noise stochastic system to exit certain bounded domains that contain stable equilibrium points or periodic orbits, from the perspective of large deviations. We anticipate that some of the methods we use to study this problem may be useful in the analysis of exit time problems for other small noise stochastic dynamical systems.

We focus on the following multidimensional small noise stochastic delay differential equation (SDDE) written in integral form:
\begin{equation}\label{eq:sdde}
	\X^\ve(t)=\X^\ve(0)+\int_0^t\f(\X_s^\ve)ds+\sqrt{\ve}\int_0^t\g(\X_s^\ve)dW(s),\qquad t\geq0.
\end{equation}
Here $\delay>0$ is the length of the finite delay interval, $\X^\ve$ is a continuous vector-valued process on $[-\tau,\infty)$, $\X_s^\ve=\{\X_s^\ve(u)=\X^\ve(s+u),u\in[-\delay,0]\}$ is a continuous process on $[-\delay,0]$ that tracks the history of $\X^\ve$ over the delay interval, $\f$ and $\g$ are continuous functions of these path segments, $W$ is a standard multidimensional Brownian motion, the stochastic integral with respect to $W$ is the It\^o integral, and $\ve$ is a small positive parameter discounting the noise coefficient. (The presence of the square root in \eqref{eq:sdde} is a matter of notational preference. Alternatively, one could scale the stochastic integral by $\ve$, in which case $\ve^2$ would appear in place of $\ve$ as the scaling constant in our main results, Theorems \ref{thm:uniformLDP} and \ref{thm:exit}.) The SDDE \eqref{eq:sdde} can be thought of as a small noise perturbation of the following deterministic DDE obtained by setting $\ve=0$ in \eqref{eq:sdde}:
\begin{equation}\label{eq:dde}
	\x(t)=\x(0)+\int_0^t\f(\x_s)ds,\qquad t\geq0.
\end{equation}
Here $\x$ is a continuous vector-valued function on $[-\delay,\infty)$ and $x_s$ is the function on the delay interval $[-\delay,0]$ defined by $x_s(u)=x(s+u)$ for all $u\in[-\delay,0]$. Since the coefficients $\f$ and $\g$ depend on the history of solutions of \eqref{eq:sdde} and \eqref{eq:dde} over the past $\tau$ units of time and solutions are continuous, the natural state space for \eqref{eq:sdde} and \eqref{eq:dde} is the infinite-dimensional set of continuous vector-valued functions on the delay interval $[-\delay,0]$, which we denote by $\C=C([-\delay,0],\R^d)$, where $d$ is the dimension of the vector-valued process $\X^\ve$ and the vector-valued path $\x$. Given an equilibrium point or periodic orbit in $\C$ of the DDE \eqref{eq:dde}, which we denote by $\orbit$ (see Definition \ref{def:periodic} below), a bounded domain $\domain$ in $\C$ that contains $\orbit$, and a solution $\X^\ve$ of the SDDE \eqref{eq:sdde}, let
	$$\exit^\ve=\inf\{t\geq0:\X_t^\ve\not\in\domain\}$$
denote the first time $\X^\ve$ exits $\domain$. If $\orbit$ is stable, $\domain$ is attracted to $\orbit$ in a manner we will make precise, $\X^\ve$ starts sufficiently close to $\orbit$, and $\sigma$ satisfies a uniform nondegeneracy condition (see Assumption \ref{ass:non-degenerate} below), then the expected exit time $E[\exit^\ve]$ will grow exponentially as $\ve$ converges to zero. Our main result is to obtain upper and lower bounds on the exponential rate at which $E[\exit^\ve]$ grows (see Theorem \ref{thm:exit} below). In addition, we show that when $\orbit$ is an equilibrium point in $\C$, and the domain $\domain$ is the uniform ball centered at $\orbit$, then the upper and lower bounds coincide (see Lemma \ref{lem:Vupperlower} below).

In order to obtain exit time asymptotics, we first prove a sample path large deviation principle (LDP) for solutions of the SDDE that is uniform over (initial conditions in) bounded sets (see Theorem \ref{thm:uniformLDP} below). A sample path LDP (for a fixed initial condition) provides upper and lower bounds on the exponential rate of decay, as the noise vanishes, for the probability the solution of the SDDE lies in a measurable set, and the rate of decay is expressed in terms of the large deviations rate function. Such a sample path LDP for an SDDE with additive noise was first established by Langevin, Oliva and de Oilveira \cite{Langevin1991}, and for an SDDE with multiplicative noise by Mohammed and Zhang \cite{Mohammed2006} (see also, the subsequent work by Mo and Luo \cite{Mo2013}). A \emph{uniform} sample path LDP over bounded sets provides upper and lower bounds on the rate of decay that hold uniformly over initial conditions in a bounded set, and the uniformity over bounded sets is crucially used to prove the exit time asymptotics. In the finite-dimensional stochastic differential equation setting, bounded sets are relatively compact, and the uniform sample path LDP over compact sets follows from the sample path LDP for a fixed initial condition using standard techniques. For general stochastic equations with multiplicative noise whose state spaces are not locally compact, the techniques for establishing uniform sample path LDPs over \emph{compact} sets do not readily extend to proving uniform sample path LDPs over \emph{bounded} sets. (For a stochastic equation with additive noise, the contraction principle can be used to prove a uniform LDP over bounded sets; see, e.g., the proof of \cite[Theorem 12.15]{daPrato1992}.) Nevertheless, uniform sample path LDPs over bounded sets have been shown for certain stochastic partial differential equations (SPDEs) with multiplicative noise and used to obtain exit time asymptotics \cite{Cerrai2004,Chenal1997,Sowers1992}. Moreover, Budhiraja, Dupuis and Salins \cite{Budhiraja2017} recently established a uniform sample path LDP over bounded sets for a broad class of SPDEs. Their approach uses a variational representation for (expectations of) exponential functionals of solutions of an SPDE along with weak convergence methods to prove a uniform LDP over bounded sets for a modified version of the SPDE, which is shown to imply a uniform LDP over bounded sets for the original SPDE. One limitation is their approach relies on compactness of an associated semigroup for all times $t>0$. In our SDDE setting, similar compactness type conditions generally only hold for times $t\geq\delay$, where we recall that $\delay>0$ is the length of the delay interval, and so the approach in \cite{Budhiraja2017} is not readily adapted to the SDDE setting. In this work we take a new approach, outlined below, which also uses the variational representation for exponential functionals of solutions, but does not rely on weak convergence methods. Since variational representations have been shown for exponential functionals of solutions for broad classes of stochastic equations, including those driven by finite-dimensional Brownian motions \cite{Boue1998}, infinite-dimensional Brownian motions \cite{Budhiraja2000,Budhiraja2008} and Poisson random measures \cite{Budhiraja2011}, we anticipate that the overall approach introduced here may be useful for proving uniform sample path LDPs over bounded sets for solutions of a variety of infinite-dimensional stochastic equations.

To begin with, we impose a uniform Lipschitz continuity condition on the coefficients (see Assumption \ref{ass:lip} below) that ensures strong existence and uniqueness of solutions of the SDDE. (While we impose a \emph{uniform} Lipschitz condition on the coefficients throughout this work, we explain in Remark \ref{rmk:locallip} that our main exit time asymptotics result, Theorem \ref{thm:exit}, holds under a \emph{local} Lipschitz condition on the coefficients.) This, along with the variational representation for exponential functionals of Brownian motion obtained in \cite[Theorem 3.1]{Boue1998} yields a variational representation for exponential functionals of solutions of the SDDE (see Lemma \ref{lem:varX} below). With the variational representation in hand, we prove a uniform Laplace principle over bounded sets (see Theorem \ref{thm:uniformLP} below). The Laplace principle (with fixed initial condition) establishes asymptotics of exponential functionals of the solution in terms of the large deviations rate function and has been shown by Varadhan \cite{Varadhan1966} and Dupuis and Ellis \cite[Theorem 1.2.3]{Dupuis1997} to be equivalent to the LDP (with fixed initial condition). Variational representations for exponential functionals of strong solutions (of broad classes of stochastic equations) have been used extensively in the weak convergence approach to prove Laplace principles and uniform Laplace principles over compact sets (see, e.g., \cite{Boue1998,Budhiraja2000,Budhiraja2008,Budhiraja2011,Dupuis1997}). Our proof of the uniform Laplace principle over \emph{bounded} sets contains some important distinctions from the weak convergence proof of the uniform Laplace principle over \emph{compact} sets. In the weak convergence approach one first establishes tightness of a family of random variables (over $\ve>0$ sufficiently small and initial conditions in a compact set) that appear in the variational representation and then characterizes the limit of any convergent subsequence as satisfying the Laplace principle upper and lower bounds. (See, for example, the proof of \cite[Theorem 4.3]{Boue1998}. The proof is for a fixed initial condition; however, the tightness arguments can be readily adapted to allow for initial conditions in a compact set.) In our SDDE setting, since bounded sets generally are not relatively compact, the family of random variables that appear in the variational representation is not necessarily tight. Instead, we leverage the fact that we are working with strong solutions, so we can build our family of small noise processes on a common probability space with a common driving Brownian motion. Then, using the Lipschitz continuity of the coefficients and standard stochastic estimates, we prove a key convergence result for the family of random variables that appear in the variational representation (see Lemma \ref{lem:distribution} below). The convergence result is used to prove the uniform Laplace principle over bounded sets, which is shown to imply a uniform LDP over bounded sets.

Lastly, we use the uniform LDP over bounded sets to prove our main exit time asymptotics result (Theorem \ref{thm:exit}). While this proof is structurally similar to the proof of \cite[Theorem 5.7.11]{Dembo1998}, there are several nontrivial modifications to the proof that are due to the fact that the version of the uniform LDP over bounded sets we obtain takes a slightly different form from the uniform LDP over compact sets. Furthermore, the SDDE is degenerate in the sense that the natural state space is infinite-dimensional while the driving Brownian motion is finite-dimensional. This degeneracy restricts the set of paths in $\C$ that a solution of the SDDE can follow when exiting a domain in $\C$ and leads to unresolved challenges in proving the upper and lower bounds for the exit time coincide (see Remark \ref{rmk:VupperVlower} below). However, in the case the domain is a uniform ball centered at an equilibrium point, we prove that the upper and lower bounds coincide (see Lemma \ref{lem:Vupperlower} below).

\subsection{Prior and related work}\label{sec:prior}

The study of exit time asymptotics for finite-dimensional SDEs is a classical subject in the theory of sample path large deviations, beginning with the work of Freidlin and Wentzell \cite{Ventcel1970,Ventcel1972}, which culminated in the books \cite{Freidlin1984,Freidlin2012}. There have been numerous other works related to exit time asymptotics for SDEs, including \cite{Day1990a,Day1990b,Day1992,Dupuis1986,Dupuis1987,Eizenberg1984,Eizenberg1987,Kifer1981}. In \cite[Chapter 12]{daPrato1992}, da Prato and Zabczyk detail a general approach for estimating exit time asymptotics for a class of small noise SPDEs with additive noise. As mentioned above, in \cite{Cerrai2004,Chenal1997,Sowers1992} the authors obtain exit time asymptotics for a variety of SPDEs with multiplicative noise and in \cite{Budhiraja2017} the authors develop a general approach for proving a uniform LDP over bounded sets for a broad class of SPDEs with multiplicative noise and compact semigroups.

There has been limited work on exit time asymptotics for SDDEs, especially those with multiplicative noise. Langevin, Oliva and de Oilveira \cite{Langevin1991} consider exit time asymptotics for SDDEs with additive noise and analyze the quasipotential (see definition \eqref{eq:quasipotential} below) associated with an asymptotically stable equilibrium point of the corresponding DDE. The proof of the exit time asymptotics in the case of additive noise relies on the contraction principle to prove a uniform LDP over bounded sets, and the method does not extend to the case of multiplicative noise. As stated above, Mohammed and Zhang \cite{Mohammed2006} prove a sample path LDP for time-inhomogeneous SDDEs with multiplicative noise and fixed initial condition in the case that $\f$ and $\g$ depend only on time, the current state and the delayed state, i.e., $\f(\X_s^\ve)=f(s,\X^\ve(s),\X^\ve(s-\delay))$ and $\g(\X_s^\ve)=g(s,\X^\ve(s),\X^\ve(s-\delay))$ for suitable functions $f$ and $g$ (see also, the work of Mo and Luo \cite{Mo2013}). We extend their result (in the case of time-homogeneous coefficients) by proving a uniform LDP over bounded sets and also allowing the coefficients to depend on the entire history of the process over the delay interval, not just the current state and delayed state. Lastly, we mention the work of Azencott, Geiger and Ott \cite{Azencott2016} who consider a linear SDDE with additive noise as a local approximation of a nonlinear SDDE and develop methods for efficient numerical computation of the rate function.

\subsection{Outline}\label{sec:outline}

The remainder of this work is organized as follows. Precise definitions for a solution of the small noise SDDE and a solution of the related DDE are given in Section \ref{sec:delay}. The definition of the rate function and our main results on the uniform sample path LDP over bounded sets and exit time asymptotics for the SDDE are presented in Section \ref{sec:main}. Some useful properties of the rate function, including compactness of level sets, are proved in Section \ref{sec:rate}. The proof of the uniform sample path LDP over bounded sets is given in Section \ref{sec:proofLDP}. The proof of the exit time asymptotics for the SDDE is given in Section \ref{sec:proofexit}.

\subsection{Notation}

Let $\R=(-\infty,\infty)$ denote the real numbers. For $r\in\R$, we say $r$ is positive (resp.\ negative, nonnegative, nonpositive) if $r>0$ (resp.\ $r<0$, $r\geq0$, $r\leq 0$). For $r,s\in\R$, we let $r\wedge s=\min(r,s)$ and $r\vee s=\max(r,s)$. For an integer $d\geq 2$, let $\R^d$ denote $d$-dimensional Euclidean space. For a column vector $\nu\in\R^d$, let $\nu^i$ denote its $i$th component, for $i=1,\dots,d$, and let $|\nu|$ denote its Euclidean norm. For positive integers $d$ and $m$, let $\mathbb{M}^{d\times m}$ denote the set of $d\times m$ matrices with real entries. Given a matrix $M\in\mathbb{M}^{d\times m}$, let $M'\in\mathbb{M}^{m\times d}$ denote the transpose of $M$ and $|M|$ denote its Frobenius norm. Given a Polish space $S$, we let $\B(S)$ denote the $\sigma$-algebra of Borel sets in $S$. The following limits will be useful throughout this work (see, e.g., \cite[Lemma 1.2.15]{Dembo1998}): If $\{\ve_n\}_{n=1}^\infty$, $\{a_n\}_{n=1}^\infty$ and $\{b_n\}_{n=1}^\infty$ are sequences of positive real numbers such that $\ve_n\to0$ as $n\to\infty$, then
	\be\label{eq:loglimit}\limsup_{n\to\infty}\ve_n\log\left(a_n+b_n\right)=\max\left(\limsup_{n\to\infty}\ve_n\log a_n,\limsup_{n\to\infty}\ve_n\log b_n\right).\ee
Suppose $B>0$. Then \eqref{eq:loglimit} implies
	\be\label{eq:loglimit1}\liminf_{n\to\infty}\ve_n\log(a_n+e^{B/\ve_n})=\max\lb\liminf_{n\to\infty}\ve_n\log a_n,B\rb.\ee

For a closed interval $I$ in $\R$ and a positive integer $d$, let $C(I,\R^d)$ denote the space of continuous functions from $I$ into $\R^d$. We endow $C(I,\R^d)$ with the topology of uniform convergence on compact intervals in $I$. This is a Polish space. Given $x\in C(I,\R^d)$ and a compact interval $J\subset I$, we define the finite supremum norm of $x$ over $J$ by
	$$\norm{x}_J=\sup_{t\in J}|x(t)|.$$
For a closed interval $I$ in $\R$, a real number $p\geq1$ and a positive integer $m$, let $L^p(I,\R^m)$ denote the Banach space of Lebesgue measurable functions $f$ from $I$ to $\R^m$ with finite $L^p$-norm: 
	$$\norm{f}_{L^p(I,\R^m)}=\lb\int_I|f(s)|^pds\rb^{\frac{1}{p}},$$
where functions that are equal almost everywhere are identified. For $T>0$ we say that a sequence $\{f_n\}_{n=1}^\infty$ in $L^2([0,T],\R^m)$ converges to $f\in L^2([0,T],\R^m)$ in the weak topology if
	$$\lim_{n\to\infty}\int_0^Tf_n(s)g(s)ds=\int_0^Tf(s)g(s)\quad\text{for all }g\in L^2([0,T],\R^m).$$
For $T,N>0$, we let 
	\be\label{eq:LN2}L_N^2([0,T],\R^m)=\lcb f\in L^2([0,T],\R^m):\int_0^T|f(s)|^2ds\leq N\rcb.\ee
When equipped with the weak topology, $L_N^2([0,T],\R^m)$ is metrizable as a compact Polish space (see \cite[Theorem III.1]{Kolmogorov1957}).

Throughout this work we let $\delay>0$ denote a fixed \emph{delay}. For $T\geq0$ we let $d_T(\cdot,\cdot)$ denote the metric on $C([-\delay,T],\R^d)$ induced by the uniform norm $\norm{\cdot}_{[-\delay,T]}$. As noted in Section \ref{sec:overview} above, when $T=0$ we use the abbreviation $\C=C([-\delay,0],\R^d)$, which is the natural state space for solutions of \eqref{eq:sdde} and \eqref{eq:dde}. For a subset $\A\subset\C$ and $\mu>0$ let
\begin{align*}
	\ball(\A,\mu)&=\{\phi\in\C:d_0(\A,\phi)<\mu\},\\
	\sphere(\A,\mu)&=\{\phi\in\C:d_0(\A,\phi)=\mu\}.
\end{align*}
Given a closed interval $I$ of the form $[-\delay,\infty)$ or $[-\delay,T]$ for some $T>0$, a path $x\in C(I,\R^d)$ and a nonnegative time $t\in I$, define $x_t\in \C$ by $x_t(s)=x(t+s)$ for $s\in[-\delay,0]$. We emphasize that $x(t)$ lies in $\R^d$ and $x_t$ lies in $\C$.

By a filtered probability space, we mean a quadruple $(\Omega,\F,\{\F_t,t\geq0\},\P)$, where $\F$ is a $\sigma$-algebra on the outcome space $\Omega$, $\P$ is a probability measure on the measurable space $(\Omega,\F)$, and $\{\F_t,t\geq0\}$ is a filtration of sub-$\sigma$-algebras of $\F$ such that $(\Omega,\F,\P)$ is a complete probability space, and for each $t\geq0$, $\F_t$ contains all $\P$-null subsets of $\F$ and $\F_t=\cap_{s>t}\F_s$. We let $E$ denote expectation under $P$. Given two $\sigma$-finite probability measures $P$ and $Q$ on a measurable space $(\Omega,\F)$, the notation $P\sim Q$ will mean that $P$ and $Q$ are mutually absolutely continuous, i.e., for any $A\in\F$, $P(A)=0$ if and only if $Q(A)=0$. By a continuous process we mean a process with all continuous sample paths.

For a positive integer $m$, by an $m$-dimensional standard Brownian motion, we mean a continuous process $W=\{W(t)=(W^1(t),\dots,W^m(t))',t\geq0\}$ taking values in $\R^m$ such that
\begin{itemize}
	\item[(i)] $W(0)=0$ a.s.,
	\item[(ii)] the coordinate processes, $W^1,\dots,W^m$, are independent,
	\item[(iii)] for each $i=1,\dots,m$, positive integer $n$ and $0\leq t_1<\cdots<t_n<\infty $, the increments $W^i(t_2)-W^i(t_1),W^i(t_3)-W^i(t_2),\dots,W^i(t_n)-W^i(t_{n-1})$ are independent, and
	\item[(iv)] for each $i=1,\dots,m$ and $0\leq s<t<\infty$, $W^i(t)-W^i(s)$ is normally distributed with mean zero and variance $t-s$.
\end{itemize}

\section{Delay differential equations}\label{sec:delay}

In this section we introduce two delay equations --- a small noise stochastic equation and a corresponding deterministic equation. Recall that we are fixing $\delay>0$, which will be referred to as the delay. In addition, we fix positive integers $d$ and $m$, recall that $\C=C([-\delay,0],\R^d)$ and fix functions $\f:\C\to\R^d$ and $\g:\C\to\mathbb{M}^{d\times m}$ satisfying the following uniform Lipschitz continuity condition. 

\begin{ass}\label{ass:lip}
There exists $\lip_1>0$ such that
\begin{equation}\label{eq:lip}
	|\f(\phi)-\f(\psi)|^2+|\g(\phi)-\g(\psi)|^2\leq \lip_1\norm{\phi-\psi}_{[-\delay,0]}^2\quad\text{for all }\phi,\psi\in\C.
\end{equation} 
\end{ass}

\begin{remark}
A simple consequence of Assumption \ref{ass:lip} is that there exists $\lip_2>0$ such that
\begin{align}\label{eq:lipbound}
	|\f(\phi)|^2+|\sigma(\phi)|^2\leq\lip_2\left(1+\norm{\phi}_{[-\delay,0]}^2\right)\quad\text{for all }\phi\in\C.
\end{align}
\end{remark}

We impose a uniform Lipschitz continuity condition to ensure existence and uniqueness of strong solutions of the SDDE. In general, a local Lipschitz continuity condition with a linear growth condition is sufficient; however, for convenience we impose the uniform condition. In Remark \ref{rmk:locallip} below, we note that our main result Theorem \ref{thm:exit} on the exit time asymptotics for the SDDE is readily extended to the case of locally Lipschitz coefficients.

\subsection{Small noise stochastic delay differential equation}\label{sec:sdde}

Throughout this section we fix $\ve>0$.

\begin{definition}\label{def:sdde}
Given an $m$-dimensional Brownian motion $W=\{W(t),t\geq0\}$ on a filtered probability space $(\Omega,\F,\{\F_t,t\geq0\},\P)$, a (strong) solution of the SDDE associated with $(b,\sigma,\ve)$ is a $d$-dimensional continuous process $\X^\ve=\{\X^\ve(t),t\geq-\delay\}$ on $(\Omega,\F,\P)$ such that $\X^\ve(t)$ is $\F_0$-measurable for each $t\in[-\delay,0]$, $\X^\ve(t)$ is $\F_t$-measurable for each $t>0$, and a.s.\ \eqref{eq:sdde} holds.
\end{definition}

The natural initial condition is a $\C$-valued random element $\xi$ on $(\Omega,\F_0,\P)$.

\begin{prop}\label{prop:sddereu}
Given an $m$-dimensional Brownian motion $W=\{W(t),t\geq0\}$ on a filtered probability space $(\Omega,\F,\{\F_t,t\geq0\},\P)$ and a $\C$-valued random element $\xi$ on $(\Omega,\F_0,\P)$, there exists a unique solution $\X^\ve$ of the SDDE with initial condition $\X_0^\ve=\xi$ and driving Brownian motion $W$. Furthermore, the process satisfies the strong Markov property.
\end{prop}

\begin{remark}
Here uniqueness means that any two solutions of the SDDE on a filtered probability space $(\Omega,\F,\{\F_t,t\geq0\},P)$ with common initial condition $\xi$ and driving Brownian motion $W$ are indistinguishable.
\end{remark}

\begin{proof}
See, e.g., \cite[Theorem 2.1]{Mohammed1984} and \cite[Theorem 2.2]{Mohammed1984}.
\end{proof}

As a consequence of Proposition \ref{prop:sddereu} we have the following corollary on the existence of a measurable function that takes a Brownian motion to the solution of the SDDE. The existence of such a function is important for our proof of the sample path LDP, as shown in Section \ref{sec:variational}.

\begin{corollary}\label{cor:strong}
For $\phi\in\C$ and $T>0$, there exists a Borel measurable function
	\be\label{eq:Lambda}\Lambda_{\phi,T}^\ve:C([0,T],\R^m)\to C([-\delay,T],\R^d)\ee
such that given an $m$-dimensional Brownian motion $W$ on any filtered probability space $(\Omega,\F,\{\F_t,t\geq0\},\P)$, the process $\Lambda_{\phi,T}^\ve(W|_{[0,T]})=\{\Lambda_\phi^\ve(W|_{[0,T]})(t),t\in[-\delay,T]\}$ is the unique solution of the SDDE with initial condition $\phi$ and driving Brownian motion $W$, restricted to the interval $[-\delay,T]$.
\end{corollary}

\begin{proof}
Given $\phi\in\C$ and $T>0$, the existence of $\Lambda_{\phi,T}^\ve$ follows from the fact that, by Proposition \ref{prop:sddereu}, there exists a unique solution of the SDDE on any filtered probability space $(\Omega,\F,\{\F_t,t\geq0\},\P)$ that supports an $m$-dimensional Brownian motion. In particular, by taking the canonical set up where $(\Omega,\F,\P)$ is $m$-dimensional Wiener space, $W=\{W(\omega,t)=\omega(t),\omega\in\Omega,t\geq0\}$ is the coordinate process and $\{\F_t,t\geq0\}$ is the $\P$-augmented filtration generated by $W$, the existence of the measurable map follows via a standard method. For a detailed outline of this method, we refer the reader to \cite[Chapter V.10]{Rogers2000a}.
\end{proof}

\emph{Throughout the remainder of this work we fix an $m$-dimensional Brownian motion $W$ on a filtered probability space $(\Omega,\F,\{\F_t,t\geq0\},\P)$.}

\begin{notation}
Given $\phi\in\C$ we write $\X^{\ve,\phi}$ to denote the unique solution of the SDDE with initial condition $\X_0^{\ve,\phi}=\phi$ and driving Brownian motion $W$.
\end{notation}

\subsection{Deterministic delay differential equation}

\begin{definition}\label{def:dder}
A solution of the DDE associated with $\f$ is a continuous function $\x\in C([-\delay,\infty),\R^d)$ that satisfies \eqref{eq:dde}.
\end{definition}

Under Assumption \ref{ass:lip}, for each $\phi\in \C$ there exists a unique solution of the DDE with initial condition $\phi$ (see, e.g., \cite[Theorem 2.3]{Hale1993}).

\begin{notation}
Given $\phi\in\C$, we let $\x^\phi$ denote the unique solution of the DDE with initial condition $\phi$.
\end{notation}

\begin{remark}\label{rmk:xdiff}
It follows from \eqref{eq:dde}, the continuity of the function $t\to\x_t$ from $[0,\infty)$ to $\C$ and the continuity of $\f$ that any solution $\x$ of the DDE is continuously differentiable on $(0,\infty)$ and its derivative satisfies $\frac{d\x(t)}{dt}=\f(\x_t)$ for all $t>0$.
\end{remark}

\begin{definition}\label{def:periodic}
A solution $\xsops$ of \eqref{eq:dde} is called \emph{periodic} with period $\period>0$ if $\xsops(t+\period)=\xsops(t)$ for all $t\geq-\delay$. Given a periodic solution $\xsops$ of \eqref{eq:dde}, we define its \emph{orbit} $\orbit$ in $\C$ by $\orbit=\{\xsops_t,t\in[0,\period)\}$.
\end{definition}

\begin{remark}
Given a periodic solution $\xsops$ with period $\period>0$, observe that $\xsops$ is also periodic with period $k\period$ for any positive integer $k$. Thus, the period is not unique; however, the orbit $\orbit$ in $\C$ is unique.
\end{remark}

\begin{definition}
A vector $\nu^\ast\in\R^d$ is called an \emph{equilibrium point} of \eqref{eq:dde} if the constant function $\xsops\in C([-\delay,\infty),\R^d)$ given by $\xsops(\cdot)\equiv\nu^\ast$ is a solution of the DDE.
\end{definition}

\begin{remark}
Suppose $\nu^\ast\in\R^d$ is an equilibrium point of \eqref{eq:dde}. Then $\f(\phi^\ast)=0$, where the constant function $\phi^\ast\in\C$ is given by $\phi^\ast(\cdot)\equiv\nu^\ast$. In addition, for any $\period>0$, $\xsops(\cdot)\equiv\nu^\ast$ is a periodic solution of \eqref{eq:dde} with period $\period$ and orbit $\orbit=\{\phi^\ast\}$,
\end{remark}

\begin{definition}\label{def:stable}
Given a periodic solution $\xsops$ of \eqref{eq:dde}, we say the orbit $\orbit$ of $\xsops$ is \emph{stable} if for each $\delta>0$ there exists $\mu\in(0,\delta]$ such that
	\be\label{eq:stable}\x_t^\phi\in\ball(\orbit,\delta)\quad\text{for all }\phi\in\ball(\orbit,\mu),\;t\geq0.\ee
\end{definition}

\begin{definition}\label{def:attracted}
Given a periodic solution $\xsops$ of \eqref{eq:dde} and a domain $\domain$ in $\C$ that contains the orbit $\orbit$ of $\xsops$, we say $\domain$ is \emph{uniformly attracted} to $\orbit$ if for each $\delta>0$ there exists $T>0$ such that 
	\be\label{eq:attracted}d_0(\orbit,\x_t^\phi)\leq\delta\quad\text{for all }\phi\in\domain,\; t\geq T.\ee
\end{definition}

We close this section with an example of a stable equilibrium point of a one-dimensional DDE and a stable periodic orbit of a one-dimensional DDE. We also provide examples of domains that are attracted to their respective orbits.

\begin{example}
Consider the following one-dimensional linear DDE (in differential form):
	\be\label{eq:1dlinear}\frac{d\x(t)}{dt}=-A\x(t)-B\x(t-\delay),\qquad t\geq0,\ee
where $B>A\geq0$. Then zero is an equilibrium point of the DDE and \eqref{eq:1dlinear} has characteristic equation
	\be\label{eq:characteristic}\lambda+A+Be^{-\lambda\delay}=0.\ee
Let $\theta_0$ be the unique solution in $[\pi/2,\pi)$ to $\cos\theta_0=-A/B$, and define
	$$\tau_0=\frac{\theta_0}{\sqrt{A^2+B^2}}.$$
If $\tau<\tau_0$, then every solution of the characteristic equation has negative real part and it follows that the orbit $\orbit=\{\phi^\ast\}$, where $\phi^\ast\in\C$ is given by $\phi^\ast(\cdot)\equiv0$, is stable and every bounded domain $\domain$ in $\C$ that contains $\orbit$ is uniformly attracted to $\orbit$ (see, e.g., \cite[Theorem 4.3]{Smith2011}).
\end{example}

\begin{definition}
Suppose $d=1$ and $\xsops$ is a periodic solution of \eqref{eq:dde}. We say $\xsops$ is a \emph{slowly oscillating periodic solution} if there exist $-\delay\leq z_0<z_1<z_2$ such that $z_1-z_0>\delay$, $z_2-z_1>\delay$, $z_2-z_0=\period$, $\xsops(t)>0$ for all $z_0<t<z_1$, and $\xsops(t)<0$ for all $z_1<t<z_2$.
\end{definition}

\begin{example}
Consider the following one-dimensional nonlinear DDE:
	\be\label{eq:ddeSOPS}\frac{d\x(t)}{dt}=f(x(t-\delay)),\qquad t\geq0,\ee
where $f:\R\to\R$ is a continuously differentiable function with $rf(r)<0$ for all $r\neq0$, $f(r)\to a$ as $r\to-\infty$ for some $a>0$, $f(r)\to-b$ as $r\to\infty$ for some $b>0$, $f'\in L^1(\R,\R)$ and $rf'(r)\to 0$ as $r\to\pm\infty$. By \cite[Theorem 1]{Xie1991}, there exists $\delay_1>0$ such that for all $\delay\geq\delay_1$ there exists a unique (up to time translation) slowly oscillating periodic solution $\xsops$ of \eqref{eq:ddeSOPS}, its orbit $\orbit$ is stable and for every $\delta>0$ sufficiently small, the domain $\ball(\orbit,\delta)$ is uniformly attracted to $\orbit$.
\end{example}

\section{Main results}\label{sec:main}

In this section we summarize our main results on the small noise asymptotics for solutions of the SDDE.

\subsection{The rate function}\label{sec:mainrate}

In this section we introduce the rate function and provide conditions under which the rate function can be explicitly evaluated. In Section \ref{sec:rate} we prove some useful properties of the rate function, including compactness of level sets.

Given $T>0$ and $\x\in C([-\delay,T],\R^d)$, let $U_T(\x)$ denote the (possibly empty) set of $u$ in $L^2([0,T],\R^m)$ such that
	\be\label{eq:xv}\x(t)=\x(0)+\int_0^tb(\x_s)ds+\int_0^t\sigma(\x_s)u(s)ds\quad\text{for all }t\in[0,T].\ee
For $\phi\in \C$ and $T>0$, define the \emph{rate function} $I_T^\phi:C([-\delay,T],\R^d)\to[0,\infty]$ by
\begin{equation}\label{eq:rate}
	I_T^\phi(\x)=
	\begin{cases}
		{\displaystyle\inf_{u\in U_T(\x)}\frac{1}{2}\int_0^T}|u(s)|^2ds&\text{if }\x_0=\phi\text{ and }U_T(\x)\neq\emptyset,\\[3ex]
		\newline\infty&\text{otherwise}.
	\end{cases}
\end{equation}

\begin{remark}\label{rmk:ratexphi}
Given $\phi\in\C$ it follows from \eqref{eq:dde} and \eqref{eq:xv} that $u(\cdot)\equiv0$ lies in $U_T(x^\phi)$ and so $I_T^\phi(x^\phi)=0$, where we recall that $\x^\phi$ denotes the unique solution of the DDE with initial condition $\phi$.
\end{remark}

\begin{notation}
Given $T>0$ and $\x\in C([-\delay,T],\R^d)$, we write $I_T(\x)$ to denote $I_T^{\x_0}(\x)$. It follows from \eqref{eq:rate} that $I_T(\x)$ satisfies
\begin{equation}\label{eq:rate1}
	I_T(\x)=
	\begin{cases}
		{\displaystyle\inf_{u\in U_T(\x)}\frac{1}{2}\int_0^T}|u(s)|^2ds&\text{if }U_T(\x)\neq\emptyset,\\[3ex]
		\newline\infty&\text{otherwise}.
	\end{cases}
\end{equation}
\end{notation}

\begin{notation}
Given $\phi\in\C$, $T>0$ and a subset $A\subset C([-\delay,T],\R^d)$, we let $I_T^\phi(A)=\inf_{\x\in A}I_T^\phi(\x)$ and $I_T(A)=\inf_{\x\in A}I_T(\x)$.
\end{notation}

In general, the variational form \eqref{eq:rate} of the rate function is difficult to explicitly evaluate. However, when $m=d$ and the following uniform ellipticity condition holds, we can explicitly evaluate the variational form.

\begin{ass}\label{ass:non-degenerate}
The diffusion coefficient $a=\sigma\sigma'$ is uniformly elliptic, i.e., there exists $c>0$ such that $\nu'a(\phi)\nu\geq c|\nu|^2$ for all $\phi\in\C$ and $\nu\in\R^d$.
\end{ass}

\begin{remark}\label{rmk:non-degenerate}
Under Assumption \ref{ass:non-degenerate}, since $a$ is continuous and uniformly elliptic, it follows from standard arguments that $a^{-1}$ is well-defined, continuous and uniformly bounded on $\C$. Thus, if $m=d$, then $\sigma^{-1}$ is well-defined and given by $\sigma^{-1}=\sigma'a^{-1}$.
\end{remark}

\begin{lemma}\label{lem:rateAC}
Suppose $m=d$ and Assumption \ref{ass:non-degenerate} holds. Then $I_T(\x)=J_T(\x)$ for all $\x\in C([-\delay,T],\R^d)$, where $J_T:C([-\delay,T]:\R^d)\to[0,\infty]$ is given by
\begin{align*}
	J_T(\x)=
	\begin{cases}
		{\displaystyle\int_0^T\Lambda(\x_s,\dotx(s))ds}&\text{if }\x\text{ is absolutely continuous on }[0,T],\\[3ex]
		\newline\infty&\text{otherwise}.
	\end{cases}
\end{align*}
Here $\Lambda:\C\times\R^d\to\R_+$ is the continuous function defined by
	\be\label{eq:Lphir}\Lambda(\phi,\nu)=\frac{1}{2}(\f(\phi)-\nu)'(a(\phi))^{-1}(\f(\phi)-\nu)\quad\text{for }(\phi,\nu)\in \C\times\R^d,\ee
and $\dotx\in L^1([0,T],\R^d)$ is such that 
	\be\label{eq:xdotx}\x(t)=\x(0)+\int_0^t\dotx(s)ds\quad\text{for all }t\in[0,T].\ee
\end{lemma}

\begin{remark}
Suppose $\x\in C([-\delay,T],\R^d)$ is absolutely continuous. Since $s\to\x_s$ is a continuous function from $[0,T]$ to $\C$, $\dot\x:[0,T]\to\R^d$ is Lebesgue measurable and $\Lambda:\C\times\R^d\to\R_+$ is continuous, it follows that the function $s\to \Lambda(\x_s,\dot\x(s))$ from $[0,T]$ to $\R_+$ is Lebesgue measurable. However, the function $s\to \Lambda(\x_s,\dot\x(s))$ need not be integrable, in which case we adopt the convention that $J_T(\x)$ is infinite.
\end{remark}

The proof of Lemma \ref{lem:rateAC} is given in Section \ref{sec:rateAC}.

\subsection{Uniform large deviation principle}

Throughout this section we fix $T>0$. Given a closed set $F\subset C([-\delay,T],\R^d)$ and $\eta\geq0$, define the enlarged closed set $F^\eta\subset C([-\delay,T],\R^d)$ by
	\be\label{eq:Feta}F^\eta=\lcb x\in C([-\delay,T],\R^d):d_T(x,F)\leq\eta\rcb,\ee
where we recall that $d_T(\cdot,\cdot)$ is the metric on $C([-\delay,T],\R^d)$ induced by the uniform norm $\norm{\cdot}_{[-\delay,T]}$. Given an open set $G\subset C([-\delay,T],\R^d)$ and $\eta\geq0$, define the shrunk open set $G_\eta\subset C([-\delay,T],\R^d)$ by
	\be\label{eq:Geta}G_\eta=\lcb x\in C([-\delay,T],\R^d):d_T(x,G^c)>\eta\rcb,\ee
where $G^c$ denotes the complement of $G$ in $C([-\delay,T],\R^d)$.

\begin{theorem}\label{thm:uniformLDP}
Suppose Assumption \ref{ass:lip} holds. Let $\K\subset\C$ be a bounded subset and $T>0$. Then the following hold:
\begin{itemize}
	\item[1.] For all closed sets $F\subset C([-\delay,T],\R^d)$,
		$$\limsup_{\ve\to0}\sup_{\phi\in\K}\ve\log\P\lb\X^{\ve,\phi}|_{[-\delay,T]}\in F\rb\leq-\lim_{\eta\to0}\inf_{\phi\in\K}I_T^\phi(F^\eta).$$
	\item[2.] For all open sets $G\subset C([-\delay,T],\R^d)$,
		$$\liminf_{\ve\to0}\inf_{\phi\in\K}\ve\log\P\lb X^{\ve,\phi}|_{[-\delay,T]}\in G\rb\geq-\lim_{\eta\to0}\sup_{\phi\in\K}I_T^\phi(G_\eta).$$
\end{itemize}
\end{theorem}

The proof of Theorem \ref{thm:uniformLDP} is given in Section \ref{sec:LDPproof}.

\subsection{Exit time asymptotics}\label{sec:exit}

Let $\xsops$ be a periodic solution of \eqref{eq:dde} with period $\period>0$ and let $\orbit=\{\xsops_t,t\in[0,\period)\}$ denote its orbit in $\C$. Define the \emph{quasipotential} $V:\C\to[0,\infty]$ associated with $\orbit$ by
	\be\label{eq:quasipotential}V(\psi)=\inf\lcb I_T(\x):T>0,\x\in C([-\delay,T],\R^d),\x_0\in\orbit,\x_T=\psi\rcb,\qquad\psi\in\C.\ee
Let $\domain$ be a bounded domain in $\C$ that contains $\orbit$. Let $\overline{\domain}$ denote the closure of $\domain$ in $\C$, $\domain^c$ denote the complement of $\domain$  in $\C$ and
	\be\label{eq:domain0}\domain^s=\lcb\phi\in\domain:\x_t^\phi\in\domain\;\forall\;t\geq0\rcb,\ee
where we recall that $\x^\phi$ denotes the solution of the DDE with initial condition $\phi\in\C$. Define
	\be\label{eq:Vbar}\overline{V}=\inf\lcb V(\psi):\psi\not\in\overline{\domain}\rcb\ee
and
	\be\label{eq:Vlower}\underline{V}=\lim_{\eta\to0}V_\eta,\ee
where
	\be\label{eq:Veta}V_\eta=\inf\{V(\psi):\psi\in\ball(\domain^c,\eta)\}.\ee
For $\ve>0$ and $\phi\in\C$, define the $\{\F_t\}$-stopping time
	\be\label{eq:exitvedelta}\exit^{\ve,\phi}=\inf\lcb t\geq0:\X_t^{\ve,\phi}\not\in\domain\rcb.\ee

\begin{theorem}\label{thm:exit}
Suppose $m=d$ and Assumptions \ref{ass:lip} and \ref{ass:non-degenerate} hold. Let $\xsops$ be a periodic solution of \eqref{eq:dde} with stable orbit $\orbit$ and let $\domain$ be a bounded domain in $\C$ that contains $\orbit$. Define $\overline{V}$ and $\underline{V}$ as in \eqref{eq:Vbar} and \eqref{eq:Vlower}, respectively. Suppose there exists $\eta_0>0$ such that $\ball(\domain,\eta_0)$ is uniformly attracted to $\orbit$. Then for all $\phi\in\domain^s$,
	\be\label{eq:Eexit}\underline{V}\leq\liminf_{\ve\to0}\ve\log \E\lsb\exit^{\ve,\phi}\rsb\leq\limsup_{\ve\to0}\ve\log \E\lsb\exit^{\ve,\phi}\rsb\leq\overline{V},\ee
and for all $\alpha>0$,
	\be\label{eq:Pexit}\lim_{\ve\to0}\P\lb e^{(\underline{V}-\alpha)/\ve}<\exit^{\ve,\phi}<e^{(\overline{V}+\alpha)/\ve}\rb=1.\ee
\end{theorem}

\begin{proof}
The theorem follows immediately from Lemmas \ref{lem:exitupper} and \ref{lem:exitlower}.
\end{proof}

\begin{remark}\label{rmk:Vfinite}
Due to the respective definitions of $\overline{V}$ and $\underline{V}$ in \eqref{eq:Vbar} and \eqref{eq:Vlower}, the definition of $V(\cdot)$ in \eqref{eq:quasipotential}, the characterization of the rate function in Lemma \ref{lem:rateAC} and the fact that $\domain$ is bounded, it readily deduced that $\overline{V}$ and $\underline{V}$ are finite.
\end{remark}

\begin{remark}\label{rmk:locallip}
Assumptions \ref{ass:lip} and \ref{ass:non-degenerate} impose a uniform Lipschitz continuity condition on the coefficients and a uniform nondegeneracy condition on the diffusion coefficient $a=\g\g'$ (on all of $\C$). However, since Theorem \ref{thm:exit} is only concerned with the process $\X^{\ve,\phi}$ up until its first exit time from the bounded domain $\domain$, the result readily extends to the case that the coefficients are uniformly Lipschitz continuous on $\overline{\domain}$ and the diffusion coefficient $a$ is uniformly nondegenerate on $\overline{\domain}$.
\end{remark}

\begin{remark}\label{rmk:VupperVlower}
One would like to show that $\underline{V}=\overline{V}$. For a general orbit $\orbit$ in $\C$ and bounded domain $\domain$ in $\C$ that contains $\orbit$, it is not clear if this equality holds. This is due to the degeneracy that arises because the state space $\C$ is infinite-dimensional while the driving Brownian motion is finite-dimensional. In particular, given an element $\phi$ on the boundary of a ``regular'' domain $\domain$ in $\C$ (i.e., $\domain$ is equal to the interior of its closure), it is possible that the solution of the SDDE with initial condition $\phi$ will almost surely remain in the domain for a positive amount of time, which is in contrast to the finite-dimensional stochastic differential equations setting. For example, suppose $d=1$, $\domain=\{\psi\in\C:\sup_{s\in[-\delay,0]}|\psi(s)|<1\}$ is the unit ball about the zero function in $\C$ and $\phi(t)=t/\delay$ for all $t\in[-\delay,0]$. Then $\phi$ lies in the boundary of $\domain$ and it is readily seen (due to the continuity of sample paths) that, for any $\ve>0$, the solution $\X^{\ve,\phi}$ of the SDDE almost surely remains in $\domain$ for a positive amount of time .
\end{remark}

\begin{lemma}\label{lem:Vupperlower}
Suppose $m=d$, Assumptions \ref{ass:lip} and \ref{ass:non-degenerate} hold, and $\nu^\ast\in\R^d$ is an equilibrium point of \eqref{eq:dde}. Define $\phi^\ast\in\C$ by $\phi^\ast(\cdot)\equiv\nu^\ast$. Let $\delta>0$, $\domain=\ball(\phi^\ast,\delta)$ and define $\overline{V}$ and $\underline{V}$ as in \eqref{eq:Vbar} and \eqref{eq:Vlower}, respectively. Then $\underline{V}=\overline{V}$.
\end{lemma}

The proof of Lemma \ref{lem:Vupperlower} is given in Section \ref{sec:Vupperlower}.

\section{Properties of the rate function}\label{sec:rate}

\subsection{Basic properties}

Given $0\leq S<T<\infty$ and $\x\in C([-\delay,T],\R^d)$, define $\x^S\in C([-\delay,T-S],\R^d)$ by
	\be\label{eq:xS}\x^S(t)=\x(S+t)\quad\text{for } t\in[-\delay,T-S].\ee

\begin{lemma}\label{lem:concat}
Let $0\leq S<T<\infty$ and $\x\in C([-\delay,T],\R^d)$. Then
	\be\label{eq:ITeqISITS} I_T(\x)=I_S(\x|_{[-\delay,S]})+I_{T-S}(\x^S).\ee
\end{lemma}

\begin{proof}
We first prove the inequality
	\be\label{eq:ITgeqISITS} I_T(\x)\geq I_S(\x|_{[-\delay,S]})+I_{T-S}(\x^S).\ee
Let $\alpha>0$. By the definition of the rate function in \eqref{eq:rate1}, there exists $u\in U_T(\x)$ such that
	\be\label{eq:vITdelta}\frac{1}{2}\int_0^T|u(s)|^2ds\leq I_T(\x)+\alpha.\ee
Since $u\in U_T(\x)$ and $S<T$, it follows that \eqref{eq:xv} holds with $S$ in place of $T$. Thus, $u|_{[0,S]}\in U_S(\x|_{[-\delay,S]})$. Therefore, by \eqref{eq:rate1},
	\be\label{eq:ISv0S}I_S(\x|_{[-\delay,S]})\leq\frac{1}{2}\int_0^S|u(s)|^2ds.\ee
Define $u^S\in L^2([0,T-S],\R^m)$ by
	\be\label{eq:vS}u^S(t)=u(S+t)\quad\text{for all }t\in[0,T-S].\ee 
By \eqref{eq:xS}, \eqref{eq:xv} and \eqref{eq:vS}, for all $t\in[0,T-S]$,
\begin{align*}
	\x^S(t)=\x(S+t)&=\x(0)+\int_0^{S+t}\f(\x_s)ds+\int_0^{S+t}\g(\x_s)u(s)ds\\
	&=\x^S(0)+\int_0^t\f(\x_s^S)ds+\int_0^t\g(\x_s^S)u^S(s)ds.
\end{align*}
Therefore, $u^S\in U_{T-S}(\x^S)$ and by \eqref{eq:rate1} and \eqref{eq:vS},
	\be\label{eq:ITSvS}I_{T-S}(\x^S)\leq\frac{1}{2}\int_0^{T-S}|u^S(s)|^2ds=\frac{1}{2}\int_S^T|u(s)|^2ds.\ee
Combining \eqref{eq:ISv0S}, \eqref{eq:ITSvS} and \eqref{eq:vITdelta}, we see that
	$$I_S(\x|_{[-\delay,S]})+I_{T-S}(\x^S)\leq\frac{1}{2}\int_0^T|u(s)|^2ds\leq I_T(\x)+\alpha.$$
Since $\alpha>0$ was arbitrary, this proves \eqref{eq:ITgeqISITS}.

Next, we prove the reverse inequality
	\be\label{eq:ITleqISITS} I_T(\x)\leq I_S(\x|_{[-\delay,S]})+I_{T-S}(\x^S).\ee
Again, let $\alpha>0$. By \eqref{eq:rate1}, there exist $u^\dagger\in U_S(\x|_{[-\delay,S]})$ and $u^\ddagger\in U_{T-S}(\x^S)$ such that
	\be\label{eq:vdaggervddagger}\frac{1}{2}\int_0^S|u^\dagger(s)|^2ds\leq I_S(\x|_{[-\delay,S]})+\frac{\alpha}{2}\quad\text{and}\quad\frac{1}{2}\int_0^{T-S}|u^\ddagger(s)|^2ds\leq I_{T-S}(\x^S)+\frac{\alpha}{2}.\ee
According to the definition of $U_S(\x|_{[-\delay,S]})$ and $U_{T-S}(\x^S)$,
	\be\label{eq:x0S}\x(t)=\x(0)+\int_0^t\f(\x_s)ds+\int_0^t\g(\x_s)u^\dagger(s)ds\quad\text{for all }t\in[0,S],\ee
and
	\be\label{eq:xSddagger}\x^S(t)=\x^S(0)+\int_0^t\f(\x_s^S)ds+\int_0^t\g(\x_s^S)u^\ddagger(s)ds\quad\text{for all }t\in[0,T-S].\ee
Define $u\in L^2([0,T],\R^m)$ by 
	\be\label{eq:udaggerddagger} u(t)=
	\begin{cases}
		u^\dagger(t)&\text{for all }t\in[0,S],\\
		u^\ddagger(t-S)&\text{for all }t\in(S,T].
	\end{cases}\ee
It follows from \eqref{eq:x0S}, \eqref{eq:xSddagger}, \eqref{eq:xS} and \eqref{eq:udaggerddagger}, that \eqref{eq:xv} holds and so $u\in U_T(\x)$. Therefore, by \eqref{eq:rate1}, \eqref{eq:udaggerddagger} and \eqref{eq:vdaggervddagger},
\begin{align*}
	I_T(\x)&\leq\frac{1}{2}\int_0^T|u(s)|^2ds\leq I_S(\x|_{[-\delay,S]})+I_{T-S}(\x^S)+\alpha.
\end{align*}
Since $\alpha>0$ was arbitrary, this proves \eqref{eq:ITleqISITS}.
\end{proof}

\begin{lemma}\label{lem:xxphi}
Let $\phi\in\C$ and $T>0$. Suppose $\x\in C([-\delay,T],\R^d)$ satisfies $I_T^\phi(\x)<\infty$. Then for all $t\in[0,T]$,
	\be\label{eq:xxphibound}\norm{\x-\x^\phi}_{[-\delay,t]}^2\leq 4I_T^\phi(\x) \lip_2\lb1+\norm{\x}_{[-\delay,t]}^2\rb t\exp(2\lip_1t^2).\ee
\end{lemma}

\begin{proof}
It follows from the definition of the rate function in \eqref{eq:rate} that $\x_0=\phi$ and, given $\alpha>0$, we can choose $u\in U_T(\x)$ so that \eqref{eq:xv} holds and
	\be\label{eq:v2ITphialpha}\frac{1}{2}\int_0^T|u(s)|^2ds<I_T^\phi(\x)+\alpha.\ee
By \eqref{eq:xv}, the fact that $\x^\phi$ satisfies \eqref{eq:dde} with $\x^\phi$ in place of $\x$, the fact that $\x_0=\x_0^\phi=\phi$, two applications of the Cauchy-Schwarz inequality, \eqref{eq:v2ITphialpha}, the Lipschitz continuity of $\f$ (Assumption \ref{ass:lip}) and the bound \eqref{eq:lipbound}, we have
\begin{align*}
	\norm{\x-\x^\phi}_{[-\delay,t]}^2&\leq2\sup_{0\leq s\leq t}\left|\int_0^s(\f(\x_r)-\f(\x_r^\phi))dr\right|^2+2\sup_{0\leq s\leq t}\left|\int_0^s\g(\x_r)u(r)dr\right|^2\\
	&\leq 2t\int_0^t|\f(\x_s)-\f(\x_s^\phi)|^2ds+4\lb I_T^\phi(\x)+\alpha\rb\int_0^t|\g(\x_s)|^2ds\\
	&\leq 2t\lip_1\int_0^t\norm{\x-\x^\phi}_{[-\delay,s]}^2ds+4\lb I_T^\phi(\x)+\alpha\rb \lip_2\lb1+\norm{\x}_{[-\delay,t]}^2\rb t.
\end{align*}
By Gronwall's inequality,
	$$\norm{\x-\x^\phi}_{[-\delay,t]}^2\leq4\lb I_T^\phi(\x)+\alpha\rb \lip_2\lb1+\norm{\x}_{[-\delay,t]}^2\rb t\exp(2\lip_1t^2).$$
Since $\alpha>0$ was arbitrary, this implies \eqref{eq:xxphibound}.
\end{proof}

\subsection{Compactness of level sets}\label{sec:levelsets}

\begin{lemma}\label{lem:levelsets}
Suppose Assumption \ref{ass:lip} holds. Let $T>0$ and $\K$ be a compact subset of $\C$. Then for each $M>0$, the level set
		\be\label{eq:CdelayTM}K_M=\lcb \x\in C([-\delay,T],\R^d):\x_0\in\K,I_T(\x)\leq M\rcb\ee
is a compact subset of $C([-\delay,T],\R^d)$.
\end{lemma}

\begin{proof}
Let $M_1>0$ be such that 
	\be\label{eq:KM1}\norm{\phi}_{[-\delay,0]}^2\leq M_1,\qquad\phi\in\K.\ee 
Fix $M>0$. Let $\alpha\in(0,1)$ be arbitrary and set $N=2M+\alpha$. Let $\x\in K_M$. By \eqref{eq:CdelayTM}, the definition of the rate function in \eqref{eq:rate1} and the definition of $U_T(\x)$ there exists $u\in L_N^2([0,T],\R^m)$ such that \eqref{eq:xv} holds. By \eqref{eq:xv}, \eqref{eq:KM1}, two applications of the Cauchy-Schwarz inequality, the definition of $L_N^2([0,T],\R^m)$ in \eqref{eq:LN2} and \eqref{eq:lipbound},
\begin{align*}
	\norm{\x}_{[-\delay,t]}^2&\leq3|\x(0)|^2+3\sup_{0\leq s\leq t}\left|\int_0^s\f(\x_r)dr\right|^2+3\sup_{0\leq s\leq t}\left|\int_0^s\g(\x_r)u(r)dr\right|^2\\
	&\leq3\norm{\x_0}_{[-\delay,0]}^2+3t\int_0^t|\f(\x_s)|^2ds+3N\int_0^t|\g(\x_s)|^2ds\\
	&\leq3M_1+3(t+N)\lip_2\int_0^t\lb 1+\norm{\x}_{[-\delay,s]}^2\rb ds.
\end{align*}
An application of Gronwall's inequality yields 
	\be\label{eq:1normx}1+\norm{\x}_{[-\delay,T]}^2\leq M_2,\ee
where $M_2=(1+3M_1)\exp\lb 3T\lip_2(T+2M+1)\rb$. Again using \eqref{eq:xv}, two applications of the Cauchy-Schwartz inequality, \eqref{eq:LN2}, \eqref{eq:lipbound} and \eqref{eq:1normx}, we have, for $0\leq s<t\leq T$,
\begin{align*}
	|\x(t)-\x(s)|^2&=2\left|\int_s^t\f(\x_r)dr\right|^2+2\left|\int_s^t\g(\x_r)u(r)dr\right|^2\\
	&\leq2(|t-s|+N)\lip_2\int_s^t\lb1+\norm{\x}_{[-\delay,r]}^2\rb dr\\
	&\leq M_3|t-s|,
\end{align*}
where $M_3=2(T+2M+1)\lip_2M_2$. Thus, on the interval $[0,T]$, $\x$ is uniformly bounded and uniformly H\"older continuous with exponent 1/2. Since the constants $M_2$ and $M_3$ depend only on $\K$, $\f$, $\g$, $M$ and $T$, it follows from the Arzela-Ascoli theorem that $\{x|_{[0,T]}:x\in K_M\}$ is relatively compact in $C([0,T],\R^d)$. Then, because $\K$ is compact in $\C$, we have $K_M$ is relatively compact in $C([-\delay,T],\R^d)$.
	
We complete the proof by showing that $K_M$ is closed. Let $\{\x^n\}_{n=1}^\infty$ be a sequence in the set \eqref{eq:CdelayTM} and $\x\in C([-\delay,T],\R^d)$ be such that $\x^n\to\x$ in $C([-\delay,T],\R^d)$ as $n\to\infty$. By \eqref{eq:rate1}, for each $n\geq1$, we can choose $u^n\in L_N^2([0,T],\R^m)$ such that
	\be\label{eq:xnvn}\x^n(t)=\x^n(0)+\int_0^t\f(\x_s^n)ds+\int_0^t\g(\x_s^n)u^n(s)ds,\qquad t\in[0,T].\ee
Since $L_N^2([0,T],\R^m)$ is a compact Polish space when equipped with the weak topology, there exists a subsequence $\{n_k\}_{k=1}^\infty$, and an element $u$ of $L_N^2([0,T],\R^m)$ such that $u^{n_k}\to u$ in the weak topology as $k\to\infty$. Letting $k\to\infty$ in \eqref{eq:xnvn} (with $n_k$ in place of $n$), we see that \eqref{eq:xv} holds. Thus, $u\in U_T(\x)$. By \eqref{eq:rate1} and the facts that $u\in L_N^2([0,T],\R^m)$ and $N=2M+\alpha$, we have
	$$I_T^\phi(\x)\leq\frac{1}{2}\int_0^T|u(s)|^2ds\leq M+\frac{\alpha}{2}.$$
Since $\alpha>0$ was arbitrary, this proves $\x\in K_M$ and thus $K_M$ is closed.
\end{proof}

\subsection{Evaluation of the variational form of the rate function}\label{sec:rateAC}

\begin{proof}[Proof of Lemma \ref{lem:rateAC}]
Fix $T>0$ and $\x\in C([-\delay,T],\R^d)$. We first show that if $I_T(\x)<\infty$, then $I_T(\x)=J_T(\x)$. Suppose $I_T(\x)<\infty$ and let $u\in U_T(\x)$. It follows from \eqref{eq:xv} that $\x$ is absolutely continuous and at the almost every $t\in[0,T]$ that $\x$ is differentiable, the derivative of $\x$ satisfies
	$$\frac{d\x(t)}{dt}=\f(\x_t)+\g(\x_t)u(t).$$
Therefore, if $\dot{\x}\in L^1([0,T],\R^d)$ is such that \eqref{eq:xdotx} holds, it follows that $\dot{\x}(t)$ satisfies, for almost every $t\in[0,T]$,
	$$\dot{\x}(t)=\f(\x_t)+\sigma(\x_t)u(t).$$
By Remark \ref{rmk:non-degenerate}, $\sigma$ is invertible and $a^{-1}=\sigma^{-1}(\sigma^{-1})'$. Rearranging the last display, we see that, for almost every $t\in[0,T]$,
\begin{align}\label{eq:vdef}
	\frac{1}{2}|u(t)|^2&=\frac{1}{2}(\g(\x_t)u(t))'(a(\x_t))^{-1}(\g(\x_t)u(t))\\ \notag
	&=\frac{1}{2}(\dotx(t)-\f(\x_t))'(a(\x_t))^{-1}(\dotx(t)-\f(\x_t))\\ \notag
	&=\Lambda(\x_t,\dotx(t)).
\end{align}
Since this holds for every $u\in U_T(\x)$, by the definition of the rate function in \eqref{eq:rate}, we have $I_T(\x)=J_T(\x)$. We are left to show that $J_T(\x)<\infty$ implies $I_T(\x)<\infty$. Suppose $J_T(\x)<\infty$. Let $\dot\x\in L^1([0,T],\R^d)$ be such that \eqref{eq:xdotx} holds. Define $u\in L^2([0,T],\R^m)$ by $u(t)=\sigma^{-1}(\x_t)(\dot{\x}(t)-b(\x_t))$ for all $t\in[0,T]$, where the square integrability of $u$ follows from the fact that $J_T(\x)$ is finite. Rearranging and substituting into \eqref{eq:xdotx} we see that \eqref{eq:xv} holds. Thus, $u\in U_T(\x)$ and so $I_T(\x)<\infty$.
\end{proof}

\section{Uniform large deviation principle}\label{sec:proofLDP}

In this section we prove Theorem \ref{thm:uniformLDP}. Throughout this section we fix $T>0$. With some abuse of notation we write $W=\{W(t),t\in[0,T]\}$ to denote the restriction of the Brownian motion to the interval $[0,T]$ and, for $\ve>0$ and $\phi\in\C$, we write $\X^{\ve,\phi}=\{\X^{\ve,\phi}(t),t\in[-\delay,T]\}$ to denote the restriction of the solution of the SDDE to the interval $[-\delay,T]$. We let $\{\F_t^W,t\in[0,T]\}$ denote the $P$-augmented filtration generated by $W$, i.e., 
	$$\F_t^W=\sigma(\{W(s),0\leq s\leq t\},\mathcal{N}),\qquad t\in[0,T],$$
where $\mathcal{N}=\{A\in\F:\P(A)=0\}$ denotes the $\P$-null sets in $\F$. We say that an $m$-dimensional process $v=\{v(t),t\in[0,T]\}$ on $(\Omega,\F_T^W,\P)$ is progressively measurable with respect to the filtration $\{\F_t^W,t\in[0,T]\}$ if, for each $t\in[0,T]$, the function from $\Omega\times[0,t]$ to $\R^m$ defined by $(\omega,s)\to v(\omega,s)$ is $\F_t^W\otimes\B([0,t])$-measurable. We use $\L^2([0,T],\R^m)$ to denote the set of $m$-dimensional processes $v=\{v(t),t\in[0,T]\}$ on $(\Omega,\F_T^W,\P)$ that are progressively measurable with respect to $\{\F_t^W,t\in[0,T]\}$ and satisfy
	$$\E\left[\int_0^T|v(s)|^2ds\right]<\infty.$$
For $N>0$, we let $\L_N^2([0,T],\R^m)$ denote the subset of processes $v$ in $\L^2([0,T],\R^m)$ that a.s.\ take values in $L_N^2([0,T],\R^m)$.

\subsection{Variational representation}\label{sec:variational}

In this section we obtain a variational representation for exponential functionals of solutions of the SDDE, which follows from the work of Bou\'e and Dupuis \cite{Boue1998} and Corollary \ref{cor:strong}. The following is a corollary of \cite[Theorem 3.1]{Boue1998}.

\begin{prop}\label{prop:variational}
Let $g:C([0,T],\R^m)\to\R$ be a bounded Borel measurable function. Then
\begin{equation}\label{eq:variational}
	\log E\left[\exp\lb-g(W)\rb\right]=-\inf_{N>0}\inf_{v\in\L_N^2([0,T],\R^m)}E\left[\frac{1}{2}\int_0^T|v(s)|^2ds+g\lb W^v\rb\right],
\end{equation}
where $W^v=\{W^v(t),t\in[0,T]\}$ is the continuous process defined by
	\be\label{eq:Wv} W^v(t)=W(t)+\int_0^t v(s)ds\quad\text{for all }t\in[0,T].\ee
\end{prop}

\begin{proof}
Fix a bounded Borel measurable function $g:C([0,T],\R^m)\to\R$. Let $M>0$ be such that $|g(w)|\leq M$ for all $w\in C([0,T],\R^m)$. According to \cite[Theorem 3.1]{Boue1998}, 
	\be\label{eq:boue}\log E\left[\exp\lb-g\lb W\rb\rb\right]=-\inf_{v\in\L^2([0,T],\R^m)}E\left[\frac{1}{2}\int_0^T|v(s)|^2ds+g\lb W^v\rb\right].\ee
Since $\L_N^2([0,T],\R^m)\subset\L^2([0,T],\R^m)$ for all $N>0$, we are left to show that
\begin{equation}\label{eq:infNinf}
	\inf_{N>0}\inf_{v\in\L_N^2([0,T],\R^m)}E\left[\frac{1}{2}\int_0^T|v(s)|^2ds+g\lb W^v\rb\right]\leq\inf_{v\in\L^2([0,T],\R^m)}E\left[\frac{1}{2}\int_0^T|v(s)|^2ds+g\lb W^v\rb\right].
\end{equation}
Let $\alpha>0$ be arbitrary and $v^\alpha\in\L^2([0,T],\R^m)$ be such that
	\be\label{eq:valpha}E\left[\frac{1}{2}\int_0^T|v^\alpha(s)|^2ds+g\lb W^{v^\alpha}\rb\right]\leq\inf_{v\in\L^2([0,T],\R^m)} E\lsb\frac{1}{2}\int_0^T|v(s)|^2ds+g\lb W^v\rb\rsb+\alpha.\ee
Set $N=2M(2M+\alpha)/\alpha$. Define the $\{\F_t^W\}$-stopping time
	$$\gamma^{\alpha,N}=\inf\lcb t\in[0,T]:\frac{1}{2}\int_0^t|v^\alpha(s)|^2ds\geq N\rcb,$$
with the convention that the infimum over the empty set is equal to $T$, and define the process $v^{\alpha,N}=\{v^{\alpha,N}(t),t\in[0,T]\}$ in $\L_N^2([0,T],\R^m)$ by 
	$$v^{\alpha,N}(t)=v^\alpha(t)1_{\{t\in[0,\gamma^{\alpha,N}]\}},\qquad t\in[0,T].$$
It follows from \eqref{eq:valpha}, \eqref{eq:boue} and the bound on $g$ that
	$$\E\lsb\frac{1}{2}\int_0^T|v^\alpha(s)|^2ds\rsb\leq2M+\alpha.$$
By Chebyshev's inequality and our choice of $N$,
	\be\label{eq:PvalphaNvalpha}\P\lb v^{\alpha,N}\neq v^\alpha\rb=\P\lb\frac{1}{2}\int_0^T|v^\alpha(s)|^2ds>N\rb\leq\frac{\alpha}{2M}.\ee
Thus, using the fact that $|v^{\alpha,N}|\leq|v^\alpha|$ holds pointwise, the bound on $g$, \eqref{eq:valpha} and \eqref{eq:PvalphaNvalpha}, we have
\begin{align*}
	\E\left[\frac{1}{2}\int_0^T|v^{\alpha,N}(s)|^2ds+g\lb W^{v^{\alpha,N}}\rb\right]&\leq\E\left[\frac{1}{2}\int_0^T|v^\alpha(s)|^2ds+g\lb W^{v^\alpha}\rb\right]+2M\P\lb v^{\alpha,N}\neq v^\alpha\rb\\
	&\leq\inf_{v\in\L^2([0,T],\R^m)}E\left[\frac{1}{2}\int_0^T|v(s)|^2ds+g\lb W^v\rb\right]+2\alpha.
\end{align*}
Since $v^{\alpha,N}\in\L_N^2([0,T],\R^m)$ and $\alpha>0$ was arbitrary, this completes the proof of \eqref{eq:infNinf}.
\end{proof}

For the following lemma, given $\phi\in\C$ and $\ve>0$, recall the Borel measurable function $\Lambda_{\phi,T}^\ve:C([0,T],\R^m)\to C([-\delay,T],\R^d)$ from Corollary \ref{cor:strong} that satisfies $P(\Lambda_{\phi,T}^\ve(W)=\X^{\ve,\phi})=1$.

\begin{lemma}\label{lem:varX}
Let $f:C([-\delay,T],\R^d)\to\R$ be a bounded Borel measurable function, $\phi\in\C$ and $\ve>0$. Then
\begin{equation}\label{eq:efXepsrep}
	\ve\log \E\left[\exp\lb-\frac{f(\X^{\ve,\phi})}{\ve}\rb\right]=-\inf_{N>0}\inf_{v\in\L_N^2([0,T],\R^m)}\E\left[\frac{1}{2}\int_0^T|v(s)|^2ds+f(\X^{\ve,v,\phi})\right],
\end{equation}
where $\X^{\ve,v,\phi}=\{\X^{\ve,v,\phi}(t),t\in[-\delay,T]\}$ is defined by $\X^{\ve,v,\phi}=\Lambda_{\phi,T}^\ve(W^{v/\sqrt{\ve}})$. In addition, for $N>0$, $\X_0^{\ve,v,\phi}=\phi$, $\X^{\ve,v,\phi}(t)$ is $\F_t^W$-measurable for each $t\in[0,T]$ and a.s.\ $\X^{\ve,v,\phi}$ satisfies, for all $t\in[0,T]$,
\begin{align}\label{eq:Xvve}
	\X^{\ve,v,\phi}(t)&=\phi(0)+\int_0^t\f(\X_s^{\ve,v,\phi})ds+\sqrt{\ve}\int_0^t\g(\X_s^{\ve,v,\phi})dW(s)+\int_0^t\g(\X_s^{\ve,v,\phi})v(s).
\end{align}
\end{lemma}

\begin{remark}
When $v\in\L^2([0,T],\R^m)$ is identically zero, $\X^{\ve,v,\phi}=\X^{\ve,\phi}$ a.s.
\end{remark}

\begin{proof}
Fix a bounded Borel measurable function $f:C([-\delay,T],\R^d)\to\R$. It follows from Corollary \ref{cor:strong} and the variational representation \eqref{eq:variational} (with $g=f\circ\Lambda_{\phi,T}^\ve$) that
\begin{align*}
	\ve\log \E\lsb\exp\lb-\frac{f(\X^{\ve,\phi})}{\ve}\rb\rsb&=-\inf_{N>0}\inf_{v\in\L^2([0,T],\R^m)}E\left[\frac{\ve}{2}\int_0^T|v(s)|^2ds+f\circ\Lambda_{\phi,T}^\ve\lb W^v\rb\right]\\
	&=-\inf_{N>0}\inf_{v\in\L^2([0,T],\R^m)}E\left[\frac{1}{2}\int_0^T|v(s)|^2ds+f\circ\Lambda_{\phi,T}^\ve(W^{v/\sqrt{\ve}})\right],
\end{align*}
and so \eqref{eq:efXepsrep} holds. Fix $N>0$ and $v\in\L_N^2([0,T],\R^m)$. Let $\ve>0$. Then
	$$\E\lsb\exp\lb\frac{1}{2\ve}\int_0^T|v(s)|^2ds\rb\rsb\leq\exp\lb\frac{N}{2\ve}\rb<\infty.$$
Thus, by Novikov's condition (see, e.g., \cite[Proposition VIII.1.15]{Revuz1999}),
	$$\lcb\mathcal{E}(t)=\exp\lb\int_0^T(v(s))'dW(s)-\frac{1}{2}\int_0^t|v(s)|^2ds\rb,\F_t^W,t\in[0,T]\rcb,$$
define a martingale. It follows from Girsanov's theorem (see, e.g., \cite[Chapter VIII, Section 1]{Revuz1999}) that the probability measure $Q$ on $(\Omega,\F_T^W)$, defined by $Q(A)=E[\mathcal{E}(T)1_A]$ for all $A\in\F_T^W$, is equivalent to $\P$, and under $Q$, $\{W^{v/\sqrt{\ve}},\F_t^W,t\in[0,T\}$ is a Brownian motion martingale. Therefore, by Corollary \ref{cor:strong}, $\X^{\ve,v,\phi}$ is $\F_t^W$-measurable for each $t\in[0,T]$, and $Q$-a.s.\ $\X_0^{\ve,v,\phi}=\phi$ and for all $t\in[0,T]$,
\begin{align*}
	\X^{\ve,v,\phi}(t)&=\phi(0)+\int_0^t\f(\X_s^{\ve,v,\phi})ds+\sqrt{\ve}\int_0^t\g(\X_s^{\ve,v,\phi})dW^{v/\sqrt{\ve}}(s)\\ \notag
	&=\phi(0)+\int_0^t\f(\X_s^{\ve,v,\phi})ds+\sqrt{\ve}\int_0^t\g(\X_s^{\ve,v,\phi})dW(s)+\int_0^t\g(\X_s^{\ve,v,\phi})v(s)ds,
\end{align*}
where we have used the definition of $W^{v/\sqrt{\ve}}$ in the second equality. Since $Q\sim\P$, it follows that $\P$-a.s.\ $\X_0^{\ve,v,\phi}=\phi$ and \eqref{eq:Xvve} holds.
\end{proof}

\begin{lemma}\label{lem:Xv}
Given $v\in\L^2([0,T],\R^m)$ and $\phi\in\C$, there is a unique continuous process $\X^{v,\phi}=\{\X^{v,\phi}(t),t\in[-\delay,T]\}$ such that $\X^{v,\phi}(t)$ is $\F_t^W$-measurable for each $t\in[0,T]$, and a.s.\ $\X_0^{v,\phi}=\phi$ and
	\be\label{eq:Xv}\X^{v,\phi}(t)=\phi(0)+\int_0^t\f(\X_s^{v,\phi})ds+\int_0^t\g(\X_s^{v,\phi})v(s)ds,\qquad t\in[0,T].\ee
In addition, a.s.\ $v\in U_T(\X^{v,\phi})$.
\end{lemma}

\begin{remark}\label{rmk:Xvrate}
Since $X^{v,\phi}$ a.s.\ satisfies \eqref{eq:Xv}, it follows that a.s.\ $v\in U_T(X^{v,\phi})$. Thus, by \eqref{eq:rate1}, a.s. $I_T(X^{v,\phi})\leq\frac{1}{2}\int_0^T|v(s)|^2ds$.
\end{remark}

\begin{remark}\label{rmk:Xvxv}
Given $\phi\in\C$, $x\in C([-\delay,T],\R^d)$ satisfying $\x_0=\phi$, and $u\in U_T(x)$, it follows from \eqref{eq:xv} and \eqref{eq:Xv} that $\X^{u,\phi}=x$.
\end{remark}

\begin{remark}
When $v\in\L^2([0,T],\R^m)$ is identically zero, $\X^{v,\phi}=\x^\phi$ a.s.
\end{remark}

\begin{proof}[Sketch of proof]
Given $v\in\L^2([0,T],\R^m)$, it follows from the definition of $\L^2([0,T],\R^m)$ that $v$ a.s.\ takes values in $L^2([0,T],\R^m)$. Let $\omega\in\Omega$ be such that $v(\omega)\in L^2([0,T],\R^m)$. Then, following a standard Picard iteration argument, which uses Assumption \ref{ass:lip}, the Cauchy-Schwarz inequality and the fact that $v(\omega)\in L^2([0,T],\R^m)$, there exists a path $\X^{v,\phi}(\omega)$ satisfying $\X_0^{v,\phi}(\omega)=\phi$ and \eqref{eq:Xv}. Uniqueness of $\X^{v,\phi}(\omega)$ follows from Assumption \ref{ass:lip}, the fact that $v(\omega)\in L^2([0,T],\R^m)$, the Cauchy-Schwarz inequality and a standard argument using Gronwall's inequality. The fact that $v(\omega)\in U_T(\X^{v,\phi}(\omega))$ follows from \eqref{eq:Xv} and the definition of the set $U_T(\X^{v,\phi}(\omega))$ given in Section \ref{sec:mainrate}. Finally, since $v$ is progressively measurable with respect to $\{\F_t^W,t\geq0\}$, it follows that $\X^{v,\phi}(t)$ is $\F_t^W$-measurable for each $t\in[0,T]$.
\end{proof}
	
\begin{lemma}\label{lem:distribution}
Let $N>0$, $\K$ be a bounded subset of $\C$ and $\{v^{\ve,\phi}\}_{\ve>0,\phi\in\K}$ be a family in $\L_N^2([0,T],\R^m)$. Then
	\be\label{eq:EXbound}\limsup_{\ve\to0}\sup_{\phi\in\K}\E\lsb\norm{\X^{\ve,v^{\ve,\phi},\phi}}_{[-\delay,T]}^2\rsb<\infty,\ee
and
	\be\label{eq:PXxzero}\lim_{\ve\to0}\sup_{\phi\in\K}\E\lsb\norm{\X^{\ve,v^{\ve,\phi},\phi}-\X^{v^{\ve,\phi},\phi}}_{[-\delay,T]}^2\rsb=0.\ee
\end{lemma}

\begin{remark}\label{rmk:distribution}
By \eqref{eq:PXxzero} with $v^{\ve,\phi}$ identically zero for each $\ve>0$, we see that
	\be\lim_{\ve\to0}\sup_{\phi\in\K}\E\lsb\norm{\X^{\ve,\phi}-\x^\phi}_{[-\delay,T]}^2\rsb=0.\ee
\end{remark}

\begin{proof} 
We first prove \eqref{eq:EXbound}. Let $\phi\in\K$. By \eqref{eq:Xvve}, the fact that $\X_0^{\ve,v^{\ve,\phi},\phi}=\phi$, three applications of the Cauchy-Schwartz inequality, the definition of $\L_N^2([0,T],\R^m)$, Doob's inequality, the It\^o isometry, \eqref{eq:lipbound} and Fubini's theorem, we have, for each $\ve>0$,
\begin{align*}
	\E\lsb\norm{\X^{\ve,v^{\ve,\phi},\phi}}_{[-\delay,t]}^2\rsb&\leq4\norm{\phi}_{[-\delay,0]}^2+4\E\lsb\sup_{0\leq s\leq t}\left|\int_0^s\f(\X_r^{\ve,v^{\ve,\phi},\phi})dr\right|^2\rsb\\
	&\qquad+4\E\lsb\sup_{0\leq s\leq t}\left|\int_0^s\g(\X_r^{\ve,v^{\ve,\phi},\phi})v^{\ve,\phi}(r)dr\right|^2\rsb\\
	&\qquad+4\ve\E\lsb\sup_{0\leq s\leq t}\left|\int_0^s\g(\X_r^{\ve,v^{\ve,\phi},\phi})dW(r)\right|^2\rsb\\
	&\leq4\norm{\phi}_{[-\delay,0]}^2+4t\E\lsb\int_0^t|\f(\X_s^{\ve,v^{\ve,\phi},\phi})|^2ds\rsb\\
	&\qquad+4(N+4\ve) \E\lsb\int_0^t|\g(\X_s^{\ve,v^{\ve,\phi},\phi})|^2ds\rsb\\
	&\leq4\norm{\phi}_{[-\delay,0]}^2+4(t+N+4\ve)\lip_2\int_0^tE\lsb \norm{1+\X^{\ve,v^{\ve,\phi},\phi}}_{[-\delay,s]}^2\rsb ds.
\end{align*}
Therefore, by Gronwall's inequality,
	\be\label{eq:1normXbound}\E\lsb1+\norm{\X^{\ve,v^{\ve,\phi},\phi}}_{[-\delay,t]}^2\rsb\leq\lb1+4\norm{\phi}_{[-\delay,0]}^2\rb\exp\lb4(t+N+4\ve)\lip_2t\rb.\ee
Taking supremums over $\phi\in\K$ on both sides and letting $\ve\to0$, we see that \eqref{eq:EXbound} holds.
	
Now let $\phi\in\K$ and $\ve>0$. By \eqref{eq:Xvve}, \eqref{eq:Xv}, three applications of the Cauchy-Schwarz inequality, the definition of $\L_N^2([0,T],\R^m)$, Doob's inequality, the It\^o isometry, Assumption \ref{ass:lip}, \eqref{eq:lipbound} and Fubini's theorem, we obtain, for $t\in[0,T]$,
\begin{align*}
	\E\lsb\norm{\X^{\ve,v^{\ve,\phi},\phi}-\X^{v^{\ve,\phi},\phi}}_{[-\delay,t]}^2\rsb&\leq\E\lsb 3\sup_{s\in[0,t]}\left|\int_0^s \lb\f(\X_r^{\ve,v^{\ve,\phi},\phi})-\f(\X_r^{v^{\ve,\phi},\phi})\rb dr\right|^2\rsb\\
	&\qquad+3\E\lsb\sup_{s\in[0,t]}\left|\int_0^s\lb\g(\X_r^{\ve,v^{\ve,\phi},\phi})-\g(\X_r^{v^{\ve,\phi},\phi})\rb v^{\ve,\phi}(s)ds\right|^2\rsb\\
	&\qquad+3\ve\E\lsb\sup_{s\in[0,t]}\left|\int_0^s\g(\X_r^{\ve,v^{\ve,\phi},\phi})dW(r)\right|^2\rsb\\
	&\leq3(t+N)\lip_1\E\lsb\int_0^t\norm{\X^{\ve,v^{\ve,\phi},\phi}-\X^{v^{\ve,\phi},\phi}}_{[-\delay,s]}^2ds\rsb\\
	&\qquad+12\ve t\lip_2\E\lsb 1+\norm{\X^{\ve,v^{\ve,\phi},\phi}}_{[-\delay,t]}^2\rsb.
\end{align*}
An application of Gronwall's inequality yields, for $t\in[0,T]$,
	$$\E\lsb\norm{\X^{\ve,v^{\ve,\phi},\phi}-\X^{v^{\ve,\phi},\phi}}_{[-\delay,t]}^2\rsb\leq 12\ve t\lip_2\E\lsb 1+\norm{\X^{\ve,v^{\ve,\phi},\phi}}_{[-\delay,t]}^2\rsb\exp(3\lip_2t(t+N)).$$
Substituting in with \eqref{eq:1normXbound}, taking supremums over $\phi\in\K$ on both sides  and letting $\ve\to0$ yields \eqref{eq:PXxzero}.
\end{proof}

\subsection{Uniform Laplace principle}

The following is the main result of this section.

\begin{theorem}\label{thm:uniformLP}
For any bounded Lipschitz continuous function $f:C([-\delay,T],\R^d)\to\R$ and bounded subset $\K\subset\C$,
	$$\limsup_{\ve\to0}\sup_{\phi\in\K}\left|\ve\log \E\lsb\exp\lb-\frac{f\lb\X^{\ve,\phi}\rb}{\ve}\rb\rsb+\inf_{\x\in C([-\delay,T],\R^d)}\left\{I_T^\phi(\x)+f(\x)\right\}\right|=0.$$
\end{theorem}

\begin{proof}
Fix a bounded Lipschitz continuous function $f:C([-\delay,T],\R^d)\to\R$ and a bounded subset $\K$ of $\C$. Let $\lip_f>0$ denote the Lipschitz constant for $f$ and let $M>0$ be sufficiently large so that 
	\be\label{eq:fboundM}|f(\x)|\leq M,\qquad\x\in C([-\delay,T],\R^d)\ee 
and 
	\be\label{eq:phiboundM}\norm{\phi}_{[-\delay,0]}\leq M,\qquad\phi\in\K.\ee

We first prove that
\begin{align}\label{eq:LPupper}
	\limsup_{\ve\to0}\sup_{\phi\in\K}\lb\ve\log \E\lsb\exp\lb-\frac{f\lb\X^{\ve,\phi}\rb}{\ve}\rb\rsb+\inf_{\x\in C([-\delay,T],\R^d)}\left\{I_T^\phi(\x)+f(\x)\right\}\rb\leq0.
\end{align}
Let $\alpha>0$ be arbitrary and set $N=2M(2M+\alpha)/\alpha$. Let $\{(\ve_n,\phi_n)\}_{n=1}^\infty$ be a sequence in $(0,\infty)\times\K$ such that $\ve_n\to0$ as $n\to\infty$. By the variational representation \eqref{eq:efXepsrep}, for each $n\geq1$ we can choose $v^n\in\mathcal{L}^2([0,T],\R^m)$ such that
\begin{align}\label{eq:vninf}
	&\ve_n\log\E\lsb\exp\lb-\frac{f\lb\X^{\ve_n,\phi_n}\rb}{\ve_n}\rb\rsb\leq-\E\lsb\frac{1}{2}\int_0^T|v^n(s)|^2ds+f(\X^{v^n,\ve_n,\phi_n})\rsb+\alpha.
\end{align}
Rearranging \eqref{eq:vninf} and using the bound \eqref{eq:fboundM} yields 
	\be\label{eq:Evn2bound}\E\left[\frac{1}{2}\int_0^T|v^n(s)|^2ds\right]\leq2M+\alpha,\qquad n\geq1.\ee
Define the $\{\F_t^W\}$-stopping time
	\be\label{eq:rhoN}\gamma^{n,N}=\inf\lcb t\in[0,T]:\frac{1}{2}\int_0^t|v^n(s)|^2ds\geq N\rcb,\ee
and define the process $v^{n,N}=\{v^{n,N}(t),t\in[0,T]\}$ in $\L_N^2([0,T],\R^m)$ by
	\be\label{eq:vnN}v^{n,N}(t)=v^n(t)1_{\{t\in[0,\gamma^{n,N}]\}}\quad\text{for all }t\in[0,T].\ee
By Chebyshev's inequality, \eqref{eq:Evn2bound} and our choice of $N$,
	\be\label{eq:PrhoN}\P\lb v^{n,N}\neq v^n\rb=\P\lb\frac{1}{2}\int_0^T|v^n(s)|^2ds>N\rb\leq\frac{\alpha}{2M}.\ee
Then by \eqref{eq:vninf}, the fact that $|v^{n,N}|\leq|v^n|$ pointwise, \eqref{eq:PrhoN} and \eqref{eq:fboundM},
\begin{align}\label{eq:vnNinf}
	&\ve_n\log\E\lsb\exp\lb-\frac{f\lb\X^{\ve_n,\phi_n}\rb}{\ve_n}\rb\rsb\leq-\E\lsb\frac{1}{2}\int_0^T|v^{n,N}(s)|^2ds+f(\X^{v^{n,N},\ve_n,\phi_n})\rsb+2\alpha.
\end{align}
For each $n\geq1$, let $\X^{v^{n,N},\phi_n}=\{\X^{v^{n,N},\phi_n}(t),t\in[-\delay,T]\}$ denote the process defined as in Lemma \ref{lem:Xv}, but with $v^{n,N}$ in place of $v$. By \eqref{eq:vnNinf}, Remark \ref{rmk:Xvrate} and the Lipschitz continuity of $f$, we obtain, for each $n\geq1$,
\begin{align*}
	\ve_n\log \E\left[\exp\lb-\frac{f\lb\X^{\ve_n,\phi_n}\rb}{\ve_n}\rb\right]&\leq-\E\left[I_T^{\phi_n}(\X^{v^{n,N},\phi_n})+f(\X^{v^{n,N},\phi_n})\right]\\
	&\qquad+\E\lsb|f(\X^{v^{n,N},\ve_n,\phi_n})-f(\X^{v^{n,N},\phi_n})|\rsb+2\alpha\\
	&\leq-\inf_{ \x\in C([-\delay,T],\R^d)}\left\{I_T^{\phi_n}(\x)+f(\x)\right\}\\
	&\qquad+\lip_f\E\lsb\norm{\X^{v^{n,N},\ve_n,\phi_n}-\X^{v^{n,N},\phi_n}}_{[-\delay,T]}\rsb+2\alpha.
\end{align*}
Rearranging and letting $n\to\infty$, it follows from Lemma \ref{lem:distribution} that
	$$\limsup_{n\to\infty}\lb\ve_n\log \E\left[\exp\lb-\frac{f\lb\X^{\ve_n,\phi}\rb}{\ve_n}\rb\right]+\inf_{ \x\in C([-\delay,T],\R^d)}\left\{I_T^{\phi_n}(\x)+f(\x)\right\}\rb\leq2\alpha.$$ 
Since this holds for every sequence $\{(\ve_n,\phi_n)\}_{n=1}^\infty$ in $(0,\infty)\times\K$ satisfying $\ve_n\to0$ as $n\to\infty$ and $\alpha>0$ was arbitrary, \eqref{eq:LPupper} holds.

Next we prove that
\begin{align}\label{eq:LPlower}
	\liminf_{\ve\to0}\inf_{\phi\in\K}\lb\ve\log E\left[\exp\lb-\frac{f(\X^{\ve,\phi})}{\ve}\rb\right]+\inf_{\x\in C([-\delay,T],\R^d)}\left\{I_T^\phi(\x)+f(\x)\right\}\rb\geq0.
\end{align}
Again, let $\alpha>0$ be arbitrary. For each $\phi\in\K$, let $\tilde{x}^\phi\in C([-\delay,T],\R^d)$ be such that
	\be\label{eq:ITphitildexphialpha} I_T^\phi(\tilde{x}^\phi)+f(\tilde{x}^\phi)\leq\inf_{x\in C([-\delay,T],\R^d)}\lcb I_T^\phi(x)+f(x)\rcb+\alpha.\ee
(Note that $\tilde{x}^\phi$ is distinct from $x^\phi$, which denotes the solution of the DDE with initial condition $\phi$.) Since $I_T^\phi(x^\phi|_{[-\delay,T]})=0$ (see Remark \ref{rmk:ratexphi}) and $f$ is bounded by $M>0$ (see \eqref{eq:fboundM} above), it follows from \eqref{eq:ITphitildexphialpha} that
\begin{align}\label{eq:ITphitildexphi}
		I_T^\phi(\tilde{x}^\phi)&\leq f(x^\phi|_{[-\delay,T]})-f(\tilde{x}^\phi)+\alpha\leq 2M+\alpha.
\end{align}
By \eqref{eq:rate} and \eqref{eq:ITphitildexphi}, we can choose $u^\phi\in U_T(\tilde{x}^\phi)$ such that
	\be\label{eq:uphiITphitildex}\frac{1}{2}\int_0^T|u^\phi(s)|^2ds\leq I_T^\phi(\tilde{x}^\phi)+\alpha\leq 2(M+\alpha).\ee
By the variational representation \eqref{eq:efXepsrep}, the fact that $u^\phi\in \L_N^2([0,T],\R^m)$ with $N=2(M+\alpha)$, and \eqref{eq:uphiITphitildex}, we see that
\begin{align}\label{eq:varLPlower}
	\ve\log\E\left[\exp\lb-\frac{f\lb\X^{\ve,\phi}\rb}{\ve}\rb\right]&=-\inf_{N>0}\inf_{v\in\L_N^2([0,T],\R^m)}\E\left[\frac{1}{2}\int_0^T|v(s)|^2ds+f\lb\X^{\ve,v,\phi}\rb\right]\\ \notag
	&\geq -\E\left[\frac{1}{2}\int_0^T|u^\phi(s)|^2ds+f(\X^{\ve,u^\phi,\phi})\right]\\ \notag
	&\geq -\lcb I_T^{\phi}(\tilde{x}^\phi)+\E\lsb f(\X^{\ve,u^\phi,\phi})\rsb\rcb-\alpha,
\end{align}
where a.s.\ $\X^{\ve,u^\phi,\phi}$ satisfies $\X_0^{\ve,u^\phi,\phi}=\phi$ and \eqref{eq:Xvve} for all $t\in[0,T]$. It follows from \eqref{eq:ITphitildexphialpha} and \eqref{eq:varLPlower} that
\begin{align}\label{eq:varLPlower1}
	&\ve\log\E\left[\exp\lb-\frac{f\lb\X^{\ve,\phi}\rb}{\ve}\rb\right]+\inf_{x\in C([-\delay,T],\R^d)}\lcb I_T^\phi(x)+f(x)\rcb\\ \notag
	&\qquad\geq\ve\log E\lsb\exp\lb-\frac{f(\X^{\ve,\phi})}{\ve}\rb\rsb+\lcb I_T^{\phi}(\tilde{x}^\phi)+f(\tilde{x}^\phi)\rcb-\alpha\\ \notag
	&\qquad\geq -E\lsb f(X^{\ve,u^\phi,\phi})-f(\tilde{x}^\phi)\rsb-2\alpha\\ \notag
	&\qquad\geq -\kappa_fE\lsb\norm{X^{\ve,u^\phi,\phi}-\tilde{x}^\phi}_{[-\delay,T]}\rsb-2\alpha.
\end{align}
By Remark \ref{rmk:Xvxv} and Lemma \ref{lem:distribution}, with $N=2(M+\alpha)$ and $v^{\ve,\phi}=u^\phi$ for each $\ve>0$ and $\phi\in\K$, we see that
	\be\label{eq:distributionuphi}\lim_{\ve\to0}\sup_{\phi\in\K}\E\lsb\norm{\X^{\ve,u^\phi,\phi}-\tilde\x^\phi}_{[-\delay,T]}^2\rsb=0.\ee
Taking infimums over $\phi$ in $\K$ and letting $\ve\to0$ in \eqref{eq:varLPlower1}, it follows from \eqref{eq:distributionuphi} that
	\be\liminf_{\ve\to0}\inf_{\phi\in\K}\lb\ve\log E\lsb\exp\lb-\frac{f(X^{\ve,\phi})}{\ve}\rb\rsb+\inf_{x\in C([-\delay,T],\R^d)}\lcb I_T^\phi(x)+ f(x)\rcb\rb\geq-2\alpha.\ee
Since $\alpha>0$ was arbitrary, this proves \eqref{eq:LPlower}. The theorem now follows from the uniform upper and lower bounds \eqref{eq:LPupper} and \eqref{eq:LPlower}. 
\end{proof}

\subsection{Proof of the uniform LDP}\label{sec:LDPproof}

The proof of Theorem \ref{thm:uniformLDP} uses the uniform Laplace principle over bounded sets proved in Theorem \ref{thm:uniformLP} and follows a similar outline to the proof of Theorem 1.2.3 in \cite{Dupuis1997}, which establishes the equivalence between the LDP and the Laplace principle. However, the proof contains some nontrivial differences that arise because we prove the LDP holds uniformly over bounded sets.

\begin{proof}[Proof of Theorem \ref{thm:uniformLDP}]
Fix a bounded subset $\K$ in $\C$ and $T>0$. We first prove part 1. Let $F$ be a closed subset of $C([-\delay,T],\R^d)$. Define the lower semicontinuous function $f:C([-\delay,T],\R^d)\to[0,\infty]$ by
	\be\label{eq:fAAc}f( x)=\begin{cases}0&\text{if } x\in F\\ \infty&\text{if } x\in F^c.\end{cases}\ee
For $j\geq1$ define
	\be\label{eq:fj}f_j( x)=j(d_T(x,F)\wedge1)\quad\text{for all }\x\in C([-\delay,T],\R^d).\ee
Then $f_j$ is bounded and Lipschitz continuous for each $j\geq1$, and $f_j$ converges to $f$ pointwise from below as $j\to\infty$. Along with \eqref{eq:fAAc}, this implies that for all $\phi\in\K$ and each $j\geq1$,
\begin{align*}
	\ve\log\P\lb\X^\ve\in F\rb&=\ve\log\E\lsb\exp\lb-\frac{f(\X^{\ve,\phi})}{\ve}\rb\rsb\leq\ve\log\E\lsb\exp\lb-\frac{f_j(\X^{\ve,\phi})}{\ve}\rb\rsb.
\end{align*}
Thus, by Theorem \ref{thm:uniformLP}, for each $j\geq1$,
\begin{align}\label{eq:XveA}
	\limsup_{\ve\to0}\sup_{\phi\in\K}\ve\log\P\lb X^{\ve}\in F\rb&\leq-\inf_{\phi\in\K}\inf_{x\in C([-\delay,T],\R^d)}\lcb I_T^\phi(x)+f_j(x)\rcb.
\end{align}
We are left to show that
	\be\label{eq:liminfjinfphiIF}\liminf_{j\to\infty}\inf_{\phi\in\K}\inf_{x\in C([-\delay,T],\R^d)}\lcb I_T^\phi(x)+f_j(x)\rcb\geq\lim_{\eta\to0}\inf_{\phi\in\K} I_T^\phi(F^\eta).\ee
If $\lim_{\eta\to0}\inf_{\phi\in\K}I_T^\phi(F^\eta)=0$, then \eqref{eq:liminfjinfphiIF} automatically holds since $I_T^\phi$ and $f_j$ are nonnegative. We assume that $\lim_{\eta\to0}\inf_{\phi\in\K}I_T^\phi(F^\eta)>0$. Since $f_j$ is zero on $F$,
	$$\inf_{x\in C([-\delay,T],\R^d)}\lcb I_T^\phi(x)+f_j(x)\rcb=\min\lb I_T^\phi(F),\inf_{x\in F^c}\lcb I_T^\phi(x)+f_j(x)\rcb\rb.$$
It suffices to show that
	\be\label{eq:liminfj}\liminf_{j\to\infty}\inf_{\phi\in\K}\inf_{x\in F^c}\lcb I_T^\phi(x)+f_j(x)\rcb\geq\lim_{\eta\to0}\inf_{\phi\in\K}I(F^\eta).\ee

First, consider the case that $\lim_{\eta\to0}\inf_{\phi\in\K}I_T^\phi(F^\eta)<\infty$. For a proof by contradiction, suppose there exists $\eta>0$, $\alpha>0$, a subsequence $\{j_k\}_{k=1}^\infty$ and a sequence $\{\x^k\}_{k=1}^\infty$ in $(F^\eta)^c$ such that for each $k\geq1$, $\x_0^k\in\K$ and
	\be\label{eq:IxkFeta}I_T(x^k)+f_{j_k}(x^k)\leq\inf_{\phi\in\K}I_T^\phi(F^\eta)-\alpha.\ee
By the uniform bound in the last display and the definition of $f_j$ in \eqref{eq:fj}, we have $d_T(x^k,F)\to0$ as $k\to\infty$. Thus, due to the definition of $F^\eta$ in \eqref{eq:Feta}, $\x^k\in F^\eta$ for all $k$ sufficiently large. However, this implies $I_T(\x^k)\geq\inf_{\phi\in\K}I_T^\phi(F^\eta)$, which contradicts \eqref{eq:IxkFeta} (since $f_j$ is nonnegative). With this contradiction thus obtained, it follows that \eqref{eq:liminfj} holds when $\lim_{\eta\to0}\inf_{\phi\in\K}I_T^\phi(F^\eta)<\infty$.

Next, consider the case that $\lim_{\eta\to0}\inf_{\phi\in\K}I_T^\phi(F^\eta)=\infty$. For a proof by contradiction, suppose there exists $\eta>0$, $M>0$, a subsequence $\{j_k\}_{k=1}^\infty$ and a sequence $\{\x^k\}_{k=1}^\infty$ in $(F^\eta)^c$ such that for each $k\geq1$, $\x_0^k\in\K$ and
	\be\label{eq:IxkM}I_T(x^k)+f_{j_k}(x^k)\leq M.\ee
By the uniform bound in the last display and the definition of $f_j$ in \eqref{eq:fj}, we have $d_T(x^k,F)\to0$ as $k\to\infty$. Thus, due to the definition of $F^\eta$ in \eqref{eq:Feta}, $\x^k\in F^\eta$ for all $k$ sufficiently large. However, this implies $I_T(\x^k)\geq\inf_{\phi\in\K}I_T^\phi(F^\eta)=\infty$, which contradicts \eqref{eq:IxkM}. With this contradiction thus obtained, it follows that \eqref{eq:liminfj} holds when $\lim_{\eta\to0}\inf_{\phi\in\K}I_T^\phi(F^\eta)=\infty$. This proves part 1 of the theorem.

We now prove part 2. Let $G$ be an open subset of $C([-\delay,T],\R^d)$. If $\lim_{\eta\to0}\sup_{\phi\in\K}I_T^\phi(G_\eta)=\infty$, we are done. We assume $\lim_{\eta\to0}\sup_{\phi\in\K}I_T^\phi(G_\eta)<\infty$. Let $\alpha>0$. Choose $\eta^\dagger>0$ such that
	\be\label{eq:Getaalpha}\sup_{\phi\in\K}I_T^\phi(G_{\eta^\dagger})\leq\lim_{\eta\to0}\sup_{\phi\in\K}I_T^\phi(G_\eta)+\alpha.\ee
Let $M>\sup_{\phi\in\K}I_T^\phi(G_{\eta^\dagger})$ and define
	$$f(x)=M\lb\frac{d_T\lb G_{\eta^\dagger},x\rb}{\eta^\dagger}\wedge1\rb.$$
Then $f$ is nonnegative, bounded above by $M$, Lipschitz continuous, and satisfies $f(x)=0$ for all $x\in G_{\eta^\dagger}$ and $f(x)=M$ for all $x\in C([-\delay,T],\R^d)$ satisfying $d_T(G_{\eta^\dagger},x)\geq\eta^\dagger$. Thus,
\begin{align*}
	\E\lsb\exp\lb-\frac{f(\X^{\ve,\phi})}{\ve}\rb\rsb&\leq e^{-M/\ve}\P(d_T(G_{\eta^\dagger},\X^{\ve,\phi})\geq\eta^\dagger)+\P(d_T(G_{\eta^\dagger},\X^\ve)<\eta^\dagger)\\
	&\leq e^{-M/\ve}+\P(d_T(G_{\eta^\dagger},\X^{\ve,\phi})<\eta^\dagger).
\end{align*}
Therefore, by \eqref{eq:loglimit1}, the last display, Theorem \ref{thm:uniformLP} and the fact that $f(x)=0$ for all $x\in G_{\eta^\dagger}$,
\begin{align*}
	\max\lb\liminf_{\ve\to0}\inf_{\phi\in\K}\ve\log\P(d_T(G_{\eta^\dagger},\X^{\ve,\phi})<\eta^\dagger),-M\rb	&\geq\liminf_{\ve\to0}\inf_{\phi\in\K}\ve\log\E\lsb\exp\lb-\frac{f(\X^{\ve,\phi})}{\ve}\rb\rsb\\
	&\geq-\sup_{\phi\in\K}\inf_{x\in C([-\delay,T],\R^d)}\{I_T^\phi(\x)+f(\x)\}\\
	&\geq-\sup_{\phi\in\K}\inf_{x\in G_{\eta^\dagger}}\lcb I_T^\phi(x)+f(x)\rcb\\
	&\geq-\sup_{\phi\in\K}I_T^\phi(G_{\eta^\dagger}).
\end{align*}
Since the ball $\{x\in C([-\delay,T],\R^d):d_T(G_{\eta^\dagger},x)<\eta^\dagger\}$ is contained in $G$ and $M>\sup_{\phi\in\K}I_T^\phi(G_{\eta^\dagger})$, it follows from the last display and \eqref{eq:Getaalpha} that
\begin{align*}
	\liminf_{\ve\to0}\inf_{\phi\in\K}\ve\log\P(\X^{\ve,\phi}\in G)&\geq\liminf_{\ve\to0}\inf_{\phi\in\K}\ve\log\P(d_T(G_{\eta^\dagger},\X^{\ve,\phi})<\eta)\\
	&\geq-\sup_{\phi\in\K}I_T^\phi(G_{\eta^\dagger})\\
	&\geq-\lim_{\eta\to0}\sup_{\phi\in\K}I_T^\phi(G_\eta)-\alpha
\end{align*}
Since $\alpha>0$ was arbitrary, this proves part 2 of the theorem.
\end{proof}

\section{Exit time asymptotics}\label{sec:proofexit}

In this section we prove Theorem \ref{thm:exit}. Throughout this section we assume $m=d$, and $\f$ and $\g$ satisfy Assumptions \ref{ass:lip} and \ref{ass:non-degenerate}. Recall that $\overline{V}$ and $\underline{V}$ are finite by Remark \ref{rmk:Vfinite}. Let $\lip_1\geq1$ be such that \eqref{eq:lip} holds and $c>0$ be the constant in Assumption \ref{ass:non-degenerate}. According to Remark \ref{rmk:non-degenerate}, there exists $M_a\geq1$ such that
	\be\label{eq:Mabound}|(a(\phi))^{-1}|\leq M_a,\qquad\phi\in\C.\ee
Let $\xsops$ be a periodic solution of \eqref{eq:dde} with period $\period>0$ and let $\orbit=\{\xsops_t,t\in[0,\period)\}$ denote its orbit in $\C$. We assume that $\xsops$ is stable (see Definition \ref{def:stable}). Let $\domain$ be a bounded domain in $\C$ that contains $\orbit$. We assume there exists $\eta_0>0$ such that $\ball(\domain,\eta_0)$ is uniformly attracted to $\orbit$ (see Definition \ref{def:attracted}). Given $\mu>0$, we let $\overline{\ball(\orbit,\mu)}=\ball(\orbit,\mu)\cup\sphere(\orbit,\mu)$ denote the closure of $\ball(\orbit,\mu)$ in $\C$.

\subsection{Preliminary estimates}

In preparation for proving Theorem \ref{thm:exit} we first establish some useful lemmas. 

\begin{lemma}\label{lem:Nmu}
Given $\alpha>0$, there are constants $\mu,T_1>0$ such that for each $\phi\in\overline{\ball(\orbit,\mu)}$ and $\psi\in\orbit$, there exist $T\leq T_1$ and $\x\in C([-\delay,T],\R^d)$ satisfying $\x_0=\phi$, $\x_T=\psi$ and $I_T(\x)\leq\alpha$.
\end{lemma}

\begin{proof}
Fix $\alpha>0$. Set $T_1=1+\delay+\period$ and
	\be\label{eq:musqrt}\mu=\sqrt{\frac{\alpha}{M_aT_1(\lip_1+1)}}.\ee
Let $\phi\in\overline{\ball(\orbit,\mu)}$ and $\psi\in\orbit$. Then there exist $t^\ast\in[0,\period)$ and $t^\dagger\in[t^\ast+1+\delay,t^\ast+1+\delay+\period)$ such that $d_0(\phi,\xsops_{t^\ast})\leq\mu$ and $\xsops_{t^\dagger}=\psi$. Define $T=t^\dagger-t^\ast\leq T_1$ and $\x\in C([-\delay,T],\R^d)$ by 
	\be\label{eq:xballlevel}
	\x(t)=
	\begin{cases}
		\phi(t)&\text{for all }t\in[-\delay,0],\\
		(1-t)\phi(0)+t\xsops(t^\ast)+\xsops(t^\ast+t)-\xsops(t^\ast)&\text{for all }t\in(0,1],\\
		\xsops(t^\ast+t)&\text{for all }t\in(1,T].
	\end{cases}
	\ee
By \eqref{eq:xballlevel} and Remark \ref{rmk:xdiff}, $\x_0=\phi$, $\x_T=\xsops_{t^\dagger}=\psi$, $x$ is absolutely continuous on $[0,T]$ and 
	$$x(t)=x(0)+\int_0^t\dotx(s)ds,\qquad t\in[0,T],$$ 
where $\dotx\in L^1([0,T],\R^d)$ is defined by
	\be\label{eq:dxdt0Tphipsi}\dotx(t)=
	\begin{cases}
		\xsops(t^\ast)-\phi(0)+\f(\xsops_{t^\ast+t})&\text{for all }t\in[0,1],\\
		\f(\xsops_{t^\ast+t})&\text{for all }t\in(1,T].
	\end{cases}\ee
By Lemma \ref{lem:rateAC}, the Lipschitz continuity of $\f$ (Assumption \ref{ass:lip}), \eqref{eq:dxdt0Tphipsi}, the fact that $T\leq T_1$, \eqref{eq:xballlevel} and the fact that $d(\phi,\xsops_{t^\ast})\leq\mu$,
\begin{align*}
	I_T(\x)&\leq\frac{1}{2}\int_0^T|(a(\x_s))^{-1}||\f(\x_s)-\f(\xsops_{t^\ast+s})+\f(\xsops_{t^\ast+s})-\dot{\x}(s)|^2ds\\
	&\leq\frac{M_a}{2}\int_0^T\left(2\lip_1\norm{\x_s-\xsops_{t^\ast+s}}_{[-\delay,0]}^2+2|\xsops(t^\ast)-\phi(0)|^2\right)ds\\
	&\leq M_aT_1(\lip_1+1)\sup_{s\in[-\delay,T]}|x(s)-x(t^\ast+s)|^2\\
	&\leq M_a{T_1}(\lip_1+1)\mu^2.
\end{align*}
The lemma then follows from our choice of $\mu$ in \eqref{eq:musqrt}.
\end{proof}

\begin{lemma}\label{lem:NmuNh}
Given $\alpha>0$, there are constants $\mu,h,T_2>0$ such that for each $\phi\in\overline{\ball(\orbit,\mu)}$ there exist $T\leq T_2$ and $\x\in C([-\delay,T],\R^d)$ such that $x_0=\phi$, $\x_T\not\in\ball(\domain,h)$ and $I_T(\x)<\overline{V}+2\alpha$.
\end{lemma}

\begin{proof}
Fix $\alpha>0$ and let $\mu,T_1>0$ be as in Lemma \ref{lem:Nmu}. By the definition of $\overline{V}$ in \eqref{eq:Vbar}, there exist $\psi\in\orbit$, $T^\dagger>0$ and $\x^\dagger\in C([-\delay,T^\dagger],\R^d)$ such that $\x_0^\dagger=\psi$, $\x_{T^\dagger}^\dagger\not\in\overline{\domain}$ and $I_{T^\dagger}(\x^\dagger)<\overline{V}+\alpha$. Since $\x_{T^\dagger}^\dagger\not\in\overline{\domain}$ and $\overline{\domain}$ is closed, we can choose $h>0$ such that $\x_{T^\dagger}^\dagger\not\in\ball(\domain,h)$. Set $T_2=T_1+T^\dagger$. Let $\phi\in\overline{\ball(\orbit,\mu)}$. By Lemma \ref{lem:Nmu}, there exist $S\leq T_1$ and $\x^\ddagger\in C([-\delay,S],\R^d)$ such that $\x_0^\ddagger=\phi$, $\x_{S}^\ddagger=\psi=\x_0^\dagger$ and $I_S(\x^\ddagger)\leq\alpha$. Let $T=S+T^\dagger\leq T_2$ and define $\x\in C([-\delay,T],\R^d)$ by
	$$\x(t)=\begin{cases}\x^\ddagger(t)&t\in[-\delay,S],\\ \x^\dagger(t-S)&t\in(S,T].\end{cases}$$
Then $\x_0=\x_0^\ddagger=\phi$, $\x_T=\x_{T^\dagger}^\dagger\not\in\ball(\domain,h)$ and, by Lemma \ref{lem:concat}, $I_T(\x)=I_S(\x^\ddagger)+I_{T^\dagger}(\x^\dagger)<\overline{V}+2\alpha$.
\end{proof}

Since $\xsops$ is stable (see Definition \ref{def:stable}) and $\domain$ is an open subset of $\C$ that contains $\orbit$, we can choose $\mu_0>0$ sufficiently small such that $\ball(\orbit,\mu_0)\subset\domain^s$, where $\domain^s$ is the subset of $\domain$ defined in \eqref{eq:domain0}. For $\ve>0$, $\phi\in\domain$ and $\mu\in(0,\mu_0)$, let 
	\be\label{eq:exitmu}\sigma_\mu^{\ve,\phi}=\inf\lcb t\geq0:\X_t^{\ve,\phi}\in{\sphere(\orbit,\mu)}\cup\domain^c\rcb.\ee
	
\begin{remark}
Since $\domain$ is a bounded subset of $\C$, $\f$ is bounded on $\overline{\domain}$ and the diffusion coefficient $a=\sigma\sigma'$ is uniformly nondegenerate on $\overline{\domain}$, it can be readily deduced that $\sigma_\mu^{\ve,\phi}$ is a.s.\ finite.
\end{remark}

\begin{lemma}\label{lem:deltaDexitmu}
Given $\mu\in(0,\mu_0)$ and $\phi\in\domain^s$,
	\be\label{eq:deltaDexitmu}\lim_{\ve\to0}\P\lb\X_{\sigma_\mu^{\ve,\phi}}^{\ve,\phi}\in{\sphere(\orbit,\mu)}\rb=1.\ee
\end{lemma}

\begin{proof}
Fix $\mu\in(0,\mu_0)$ and $\phi\in\domain^s$. By the definition of $\domain^s$ in \eqref{eq:domain0} and the fact that $\domain$ is uniformly attracted to $\orbit$ (see Definition \ref{def:attracted}), $\x_t^\phi\in\domain$ for all $t\geq0$ and 
	$$T=\inf\{t\geq0:d_0(\orbit,\x_t^\phi)\leq\mu/2\}<\infty.$$ 
Since $t\to\x_t^\phi$ is a continuous function from $[0,T]$ to $\C$, $\x_t^\phi\in\domain$ for all $t\in[0,T]$ and $\domain$ is an open subset of $\C$, we can choose $h\in(0,\mu/2)$ sufficiently small so that $\ball(\x_t^\phi,h)\subset\domain$ for all $t\in[0,T]$. Therefore, if $\norm{\X^{\ve,\phi}-\x^\phi}_{[-\delay,T]}<h$, then $\X_t^{\ve,\phi}\in\domain$ for all $t\in[0,T]$ and $\X_T^{\ve,\phi}\in{\sphere(\orbit,\mu)}$, which, by the sample path continuity of $t\to\X_t^{\ve,\phi}$ from $[0,T]$ to $\C$, imply that $\sigma_\mu^{\ve,\phi}<T$ and $\X_{\sigma_\mu^{\ve,\phi}}^{\ve,\phi}\in{\sphere(\orbit,\mu)}$. Thus, by Chebyshev's inequality and Remark \ref{rmk:distribution}, 
\begin{align*}
	\lim_{\ve\to0}P\lb\X_{\sigma_\mu^{\ve,\phi}}^{\ve,\phi}\not\in{\sphere(\orbit,\mu)}\rb&\leq\lim_{\ve\to0}\P\lb\norm{\X^{\ve,\phi}-\x^\phi}_{[-\delay,T]}\geq h\rb=0.
\end{align*}
\end{proof}

\begin{lemma}\label{lem:IphiVetaalpha}
Given $\alpha,\eta>0$ there exists $\mu\in(0,\mu_0)$ such that if $T>0$, $x\in C([-\delay,T],\R^d)$ and $x_T\in\ball(\domain^c,\eta)$, then $\inf_{\phi\in\sphere(\orbit,\mu)}I_T^\phi(\x)\geq V_{2\eta}-\alpha$.
\end{lemma}

\begin{proof}
Fix $\alpha,\eta>0$. By \eqref{eq:Veta}, \eqref{eq:Vlower} and Remark \ref{rmk:Vfinite}, we have  $V_{2\eta}\leq\underline{V}<\infty$. Set
	\be\label{eq:mudef}\beta=\min\lb\sqrt{\frac{\alpha}{2M_a(\lip_1\delay+\delay^{-1})}},\frac{2\alpha}{7M_a(2V_{2\eta}c^{-1}+\delay)},1\rb.\ee
By \eqref{eq:Mabound}, for $\phi,\psi\in\C$,
\begin{align}\label{eq:ainversecty}
	|(a(\phi))^{-1}-(a(\psi))^{-1}|&=|(a(\phi))^{-1}(a(\psi)-a(\phi))(a(\psi))^{-1}|\leq M_a^2|a(\psi)-a(\phi)|.
\end{align}
Since $a=\g\g'$, $\g$ is Lipschitz continuous (Assumption \ref{ass:lip}) and $\overline{\domain}$ is a bounded set, it follows from \eqref{eq:ainversecty} that $a^{-1}$ is Lipschitz continuous on $\overline{\domain}$. Therefore, we can choose 
	\be\label{eq:muminmu0gamma}0<\mu<\min\lb\mu_0,\frac{\beta}{2\lip_1},\frac{\eta}{2}\rb\ee 
sufficiently small such that
	\be\label{eq:ainverseunifcont}|(a(\phi))^{-1}-(a(\psi))^{-1}|\leq\beta\ee
for all $\phi,\psi\in\overline{\domain}$ satisfying $d_0(\phi,\psi)\leq 2\mu$. 

Suppose $T>0$, $x\in C([-\delay,T],\R^d)$ and $x_T\in\ball(\domain^c,\eta)$. If $\inf_{\phi\in\sphere(\orbit,\mu)}I_T^\phi(x)\geq V_{2\eta}$, we are done. Thus, we can assume that 
	\be\label{eq:ITlessVeta}\inf_{\phi\in\sphere(\orbit,\mu)}I_T^\phi(x)<V_{2\eta}<\infty.\ee 
In addition, if $x_S\in\ball(\domain^c,\eta)$ for some $S\in[0,T)$, then by Lemma \ref{lem:concat}, $\inf_{\phi\in\sphere(\orbit,\mu)} I_T(x)\geq\inf_{\phi\in\sphere(\orbit,\mu)}I_S(x|_{[-\delay,S]})$ and it suffices to show that $\inf_{\phi\in\sphere(\orbit,\mu)}I_S^\phi(x|_{[-\delay,S]})\geq V_{2\eta}-\alpha$. Therefore, without loss of generality, we can assume that $x_t\in\domain$ for all $t\in[0,T]$. By \eqref{eq:ITlessVeta} and the definition of the rate function in \eqref{eq:rate}, we must have 
	\be\label{eq:x0sphereorbitmu}x_0\in\sphere(\orbit,\mu).\ee 
Set $T^\dagger=\delay+T$, let $t^\ast\in[\delay,\delay+\period)$ be such that $d_0(\xsops_{t^\ast},x_0)=\mu$ and define $x^\dagger\in C([-\delay,T^\dagger],\R^d)$ by
	\be\label{eq:xdaggerdelayTdagger} x^\dagger(t)=\begin{cases}
		\xsops(t^\ast-\delay+t)&\text{for all }t\in[-\delay,0],\\
		\xsops(t^\ast-\delay+t)+\frac{t}{\delay}(x(0)-\xsops(t^\ast))&\text{for all }t\in(0,\delay],\\
		x(t-\delay)&\text{for all }t\in(\delay,T^\dagger].
	\end{cases}\ee
By Remark \ref{rmk:xdiff} and the periodicity of $\xsops$, $\xsops$ is continuously differentiable on $(-\delay,\infty)$ and its derivative satisfies $\frac{d\xsops(t)}{dt}=\f(\xsops_t)$ for all $t\geq0$. In addition, it follows from \eqref{eq:ITlessVeta} and Lemma \ref{lem:rateAC} that $x$ is absolutely continuous on $[0,T]$ and there exists $\dot x\in L^1([0,T],\R^d)$ such that \eqref{eq:xdotx} holds. Therefore, by \eqref{eq:xdaggerdelayTdagger}, $\x^\dagger$ is absolutely continuous on $[-\delay,T^\dagger]$ and 
	$$x^\dagger(t)=\x^\dagger(0)+\int_0^t\dot x^\dagger(s)ds,\qquad t\in[0,T^\dagger],$$
where $\dot x^\dagger\in L^1([0,T^\dagger],\R^d)$ is defined by
	\be\label{eq:dotxdaggerdelayTdagger}\dot x^\dagger(t)=\begin{cases}\f(\xsops_{t^\ast-\delay+t})+\frac{1}{\delay}(x(0)-\xsops(t^\ast))&\text{for all }t\in[0,\delay],\\
		\dot x(t-\delay)&\text{for all }t\in[\delay,T^\dagger].\end{cases}\ee
By Lemma \ref{lem:concat}, 
	\be\label{eq:xdaggerconcat}I_{T^\dagger}(x^\dagger)=I_\delay(\x^\dagger|_{[-\delay,\delay]})+I_T(x^\delay),\ee
where $x^\delay\in C([-\delay,T],\R^d)$ is defined by 
	\be\label{eq:xdelay}x^\delay(t)=x^\dagger(\delay+t),\qquad t\in[-\delay,T].\ee 
Note that by \eqref{eq:xdelay} and \eqref{eq:xdaggerdelayTdagger},
\begin{align*}
|x^\delay(t)-x(t)|&\leq|\xsops(t^\ast+t)-x(t)|+\frac{\delay+t}{\delay}|x(0)-\xsops(t^\ast)|\leq2\mu
\end{align*}
for all $t\in[-\delay,0]$, where we have used the fact that $d_0(\xsops_{t^\ast},x_0)=\mu$; and
	\be\label{eq:xdelayx}x^\delay(t)=x(t)\qquad\text{for all }t\in[0,T].\ee
Thus,
	\be\label{eq:supxdelaysxs}\sup_{s\in[-\delay,T]}|x^\delay(s)-x(s)|\leq 2\mu.\ee
By the fact that $x_T\in\ball(\domain^c,\eta)$, \eqref{eq:supxdelaysxs} and \eqref{eq:muminmu0gamma}, we have
	$$d_0(\domain^c,x_{T^\dagger}^\dagger)\leq d_0(\domain^c,x_T)+d_0(x_T,x_{T^\dagger}^\dagger)<2\eta.$$
Hence, $x_{T^\dagger}^\dagger\in\ball(\domain^c,2\eta)$. Since $x_0^\dagger=\xsops_{t^\ast-\delay}\in\orbit$ and $x_{T^\dagger}^\dagger\in\ball(\domain^c,2\eta)$, it follows from \eqref{eq:Veta} that 
	\be\label{eq:ITdaggerxdaggerVeta}I_{T^\dagger}(x^\dagger)\geq V_{2\eta}.\ee 
We claim, and prove below, that the following two inequalities hold:
\begin{itemize}
	\item[(a)] $I_\delay(x^\dagger|_{[-\delay,\delay]})<\frac{\alpha}{2}$;
	\item[(b)] $I_T(x^\delay)\leq I_T(x)+\frac{\alpha}{2}$.
\end{itemize}
Assuming the claim, then by \eqref{eq:x0sphereorbitmu}, (b), \eqref{eq:xdaggerconcat}, \eqref{eq:ITdaggerxdaggerVeta} and (a), we have
	$$\inf_{\phi\in\sphere(\orbit,\mu)}I_T^\phi(x)=I_T(x)\geq I_T(x^\delay)-\frac{\alpha}{2}\geq V_{2\eta}-\alpha.$$
We are left to prove that (a) and (b) hold. 

We first prove (a). By Lemma \ref{lem:rateAC}, \eqref{eq:Mabound}, the Lipschitz continuity of $\f$ (Assumption \ref{ass:lip}), \eqref{eq:dotxdaggerdelayTdagger}, \eqref{eq:xdaggerdelayTdagger}, the fact that $d_0(\xsops_{t^\ast},x_0)=\mu$, \eqref{eq:muminmu0gamma} and \eqref{eq:mudef}, we have
\begin{align*}
	I_\delay(x^\dagger|_{[-\delay,\delay]})&\leq\frac{1}{2}\int_0^\delay|(a(x_s^\dagger))^{-1}||\f(x_s^\dagger)-\f(\xsops_{t^\ast-\delay+s})+\f(\xsops_{t^\ast-\delay+s})-\dot x^\dagger(s)|^2ds\\
	&\leq\frac{M_a\delay}{2}\lb2\lip_1\sup_{s\in[-\delay,\delay]}|\x^\dagger(s)-\xsops(t^\ast-\delay+s)|^2+\frac{2}{\delay^2}|\x(0)-\xsops(t^\ast)|^2\rb\\
	&\leq M_a(\lip_1\delay+\delay^{-1})\mu^2\\
	&<\frac{\alpha}{2}.
\end{align*}
This proves (a). 

Next, we prove (b). It follows from \eqref{eq:supxdelaysxs}, \eqref{eq:ainverseunifcont}, the Lipschitz continuity of $\f$ and \eqref{eq:muminmu0gamma} that, for all $s\in[0,T]$,
\begin{align}\label{eq:abeta}
	|(a(x_s))^{-1}-(a(x_s^\delay))^{-1}|&\leq\beta\\ \label{eq:fbeta}
	|\f(x_s)-\f(x_s^\delay)|&\leq\beta.
\end{align}
By Lemma \ref{lem:rateAC} and \eqref{eq:xdelayx}, we have
\begin{align}\label{eq:ITxdelay}
	I_T(x^\delay)&=\frac{1}{2}\int_0^T(\f(x_s^\delay)-\dot x^\delay(s))'(a(x_s^\delay))^{-1}(\f(x_s^\delay)-\dot x^\delay(s))ds\\ \notag
	&=\frac{1}{2}\int_0^T(\f(x_s)-\dot x(s))'(a(x_s))^{-1}(\f(x_s)-\dot x(s))ds+\frac{1}{2}\int_0^{T\wedge\delay}k(s)ds\\ \notag
	&=I_T(x)+\frac{1}{2}\int_0^{T\wedge\delay}k(s)ds,
\end{align}
where, $k:[0,T]\to\R^d$ is defined, for $s\in[0,T]$, by
\begin{align*}
	k(s)&=(\f(\x_s)-\dot x(s))'(a(x_s))^{-1}(\f(\x_s^\delay)-\f(\x_s))\\
	&\qquad+(\f(\x_s)-\dot x(s))'((a(x_s^\delay)^{-1}-(a(x_s))^{-1})(\f(\x_s)-\dotx(s))\\
	&\qquad+(\f(\x_s)-\dot x(s))'((a(x_s^\delay)^{-1}-(a(x_s))^{-1})(\f(\x_s^\delay)-\f(\x_s))\\
	&\qquad+(\f(\x_s^\delay)-\f(\x_s))'(a(x_s))^{-1}(\f(x_s)-\dotx(s))\\
	&\qquad+(\f(\x_s^\delay)-\f(\x_s))'(a(x_s))^{-1}(\f(x_s^\delay)-\f(\x_s))\\
	&\qquad+(\f(x_s^\delay)-\f(x_s))'((a(x_s^\delay)^{-1}-(a(x_s))^{-1})(\f(\x_s)-\dotx(s))\\
	&\qquad+(\f(x_s^\delay)-\f(x_s))'((a(x_s^\delay)^{-1}-(a(x_s))^{-1})(\f(\x_s^\delay)-\f(\x_s)).
\end{align*}
By \eqref{eq:Mabound}, \eqref{eq:abeta}, \eqref{eq:fbeta} and \eqref{eq:mudef}, we have, for all $s\in[0,T]$,
\begin{align}\label{eq:ks}
		|k(s)|&\leq 7M_a(|\f(x_s)-\dot x(s)|^2\vee 1)\beta.
\end{align}
Note that, by Assumption \ref{ass:non-degenerate}, Lemma \ref{lem:rateAC} and \eqref{eq:x0sphereorbitmu},
\begin{align*}
	\int_0^{T\wedge\delay}(|\f(x_s)-\dot x(s)|^2\vee1)ds&\leq\int_0^{T\wedge\delay}|\f(x_s)-\dot x(s)|^2ds+\delay\\
	&\leq\frac{1}{c}\int_0^T(\f(x_s)-\dot x(s))'(a(x_s))^{-1}(\f(x_s)-\dot x(s))ds+\delay\\
	&\leq\frac{2I_T(x)}{c}+\delay.
\end{align*}
Thus, by \eqref{eq:ks}, the last display, \eqref{eq:ITlessVeta} and \eqref{eq:mudef},
\begin{align*}
	\frac{1}{2}\int_0^{T\wedge\delay}|k(s)|ds&<\frac{7}{2}M_a\lb\frac{2V_{2\eta}}{c}+\delay\rb\beta\leq\alpha.
\end{align*}
Along with \eqref{eq:ITxdelay}, this completes the proof of (b). 
\end{proof}

\begin{lemma}\label{lem:sigmamuT}
Given $\mu\in(0,\mu_0/2)$,
	$$\lim_{T\to\infty}\limsup_{\ve\to0}\sup_{\phi\in\sphere(\orbit,2\mu)}\ve\log\P\left(\sigma_\mu^{\ve,\phi}>T\right)=-\infty.$$ 
\end{lemma}

\begin{proof}
Fix $\mu\in(0,\mu_0/2)$. For $T>0$, define the closed set
	\be\label{eq:FTDB}F_T=\lcb\x\in C([-\delay,T],\R^d):\x_t\in\overline{\domain}\setminus\ball(\orbit,\mu)\text{ for all }t\in[0,T]\rcb.\ee
For $\ve>0$ and $\phi\in\sphere(\orbit,2\mu)$, it follows from the definition of $\sigma_\mu^{\ve,\phi}$ in \eqref{eq:exitmu} that
	$$\{\sigma_\mu^{\ve,\phi}>T\}\subset\{\X^{\ve,\phi}|_{[-\delay,T]}\in F_T\}.$$
Thus, by the uniform LDP upper bound stated in part 2 of Theorem \ref{thm:uniformLDP},
\begin{align*}
	\limsup_{\ve\to0}\ve\log\sup_{\phi\in\sphere(\orbit,2\mu)}\P(\sigma_\mu^{\ve,\phi}>T)	&\leq-\lim_{\eta\to0}\inf_{\phi\in\sphere(\orbit,2\mu)}I_T^\phi(F_T^\eta),
\end{align*}
where $F_T^\eta$ is the closed subset of $C([-\delay,T],\R^d)$ defined as in \eqref{eq:Feta}, but with $F_T^\eta$ and $F_T$ in place of $F^\eta$ and $F$, respectively. Thus, we are left to prove
	\be\label{eq:limTetaIinfty}\lim_{T\to\infty}\lim_{\eta\to0}\inf_{\phi\in\sphere(\orbit,2\mu)}I_T^\phi(F_T^\eta)=\infty.\ee
By \eqref{eq:Feta} and \eqref{eq:FTDB}, for all $\eta\in(0,\mu)$ and $T>0$,
	\be\label{eq:FTeta}F_T^{\eta}=\lcb x\in C([-\delay,T],\R^d):x_t\in\overline{\ball(\domain,\eta)}\setminus\ball(\orbit,\mu-\eta)\text{ for all }t\in[0,T]\rcb,\ee
where $\overline{\ball(\domain,\eta)}$ denotes the closure of $\ball(\domain,\eta)$ in $\C$. It is straightforward to check that 
\begin{itemize}
	\item[(a)] for fixed $T>0$, $\inf_{\phi\in\sphere(\orbit,2\mu)}I_T^\phi(F_T^\eta)$ is nondecreasing as $\eta\to0$, and
	\item[(b)] for fixed $\eta>0$, $\inf_{\phi\in\sphere(\orbit,2\mu)}I_T^\phi(F_T^\eta)$ is nondecreasing as $T\to\infty$.
\end{itemize}
Recall that $\eta_0>0$ is such that $\ball(\domain,\eta_0)$ is uniformly attracted to $\orbit$. Fix $\eta_1\in(0,\mu\wedge\eta_0)$. By (a), in order to prove \eqref{eq:limTetaIinfty}, it suffices to show that
	\be\label{eq:limTIinfty}\lim_{T\to\infty}\inf_{\phi\in\sphere(\orbit,2\mu)}I_T^\phi(F_T^{\eta_1})=\infty.\ee
For a proof by contradiction, suppose there exists $M>0$ such that
	\be\label{eq:ITM}\inf_{\phi\in\sphere(\orbit,2\mu)}I_T^\phi(F_T^{\eta_1})\leq M\quad\text{for all }T>0.\ee
Since $\ball(\domain,\eta_0)$ is uniformly attracted to $\orbit$, we can choose $T^\dagger>0$ sufficiently large such that
	\be\label{eq:Tgammainverse}d_0(\orbit,\x_{T^\dagger}^\phi)<\frac{\mu-\eta_1}{2},\qquad\phi\in\ball(\domain,\eta_0).\ee
Let $M_1>0$ be sufficiently large such that 
	\be\label{eq:M1bound}\norm{\phi}_{[-\delay,0]}\leq M_1,\qquad\phi\in\ball(\domain,\eta_0).\ee 
Let $\lip_2>0$ be as in \eqref{eq:lipbound} and $n$ be a positive integer satisfying 
	$$n\geq\frac{16M\lip_2(1+M_1^2)T^\dagger e^{2\lip_1(T^\dagger)^2}}{\mu-\eta_1}.$$ 
By \eqref{eq:ITM}, there exists $\x\in F_{nT^\dagger}^{\eta_1}$ such that
	$$I_{nT^\dagger}(\x)\leq 2M.$$
Define the functions $\x^1,\dots,\x^n$ on $[-\delay,T]$ by
	\be\label{eq:xk}\x^k(t)=\x(kS+t),\qquad t\in[-\delay,T],\qquad k=1,\dots,n.\ee
Since $x\in F_{nT^\dagger}^{\eta_1}$, it follows from \eqref{eq:xk} and \eqref{eq:FTeta} that $\x^k\in F_{T^\dagger}^{\eta_1}$ for each $k=1,\dots,n$. Then by $n-1$ applications of Lemma \ref{lem:concat}, we have
	$$2M\geq I_{nT^\dagger}(\x)=\sum_{k=1}^{n}I_{T^\dagger}\lb\x^k\rb.$$
Thus, there exists $l\in\{1,\dots,n\}$ such that
	$$I_{T^\dagger}\lb\x^l\rb\leq\frac{2M}{n}\leq\frac{\mu-\eta_1}{2}\frac{1}{4\lip_2(1+M_1^2)T^\dagger e^{2\lip_1(T^\dagger)^2}}.$$
Since $x^l\in F_{T^\dagger}^{\eta_1}$, it follows from \eqref{eq:FTeta} that $x_t^l\in\ball(\domain,\eta_0)$ for all $t\in[0,T^\dagger]$. Set $\phi=x_0^l\in\ball(\domain,\eta_0)$. Then by Lemma \ref{lem:xxphi}, the last display and \eqref{eq:M1bound},
\begin{align}\label{eq:xxl}
	\norm{\x^l-\x^\phi}_{[-\delay,T^\dagger]}^2\leq4I_{T^\dagger}(\x^l)\lip_2\lb1+\norm{x^l}_{[-\delay,T^\dagger]}^2\rb T^\dagger e^{2\lip_1(T^\dagger)^2}\leq\frac{\mu-\eta_1}{2}.
\end{align}
Therefore, by \eqref{eq:xxl}, the fact that $\phi\in\ball(\domain,\eta_0)$ and \eqref{eq:Tgammainverse},
	$$d_0(\orbit,\x_T^l)\leq d_0(\x_T^l,\x_T^\phi)+d_0(\orbit,\x_T^\phi)<\mu-\eta_1,$$
which contradicts the fact that $\x^l\in F_T^\eta$. With the contradiction thus obtained, it follows that \eqref{eq:limTIinfty} holds. This completes the proof of the lemma.
\end{proof}

\begin{lemma}\label{lem:phiS2muVbar}
Given $\alpha>0$, there exists $\mu\in(0,\mu_0/2)$ such that
	\be\label{eq:phiS2muVbar}\limsup_{\ve\to0}\sup_{\phi\in \sphere(\orbit,2\mu)}\ve\log\P\lb\X_{\sigma_\mu^{\ve,\phi}}^{\ve,\phi}\in\domain^c\rb\leq-\underline{V}+2\alpha.\ee
\end{lemma}

\begin{proof}
Fix $\alpha>0$. By the definition of $\underline{V}$ in \eqref{eq:Vlower}, we can choose $\eta^\dagger>0$ such that $V_{2\eta^\dagger}\geq\underline{V}-\alpha$. Let $\mu\in(0,\mu_0/2)$ be as in Lemma \ref{lem:IphiVetaalpha} (with $\eta^\dagger$ in place of $\eta$). By Lemma \ref{lem:sigmamuT}, we can choose $T>0$ such that
	\be\label{eq:sigmaTV}\limsup_{\ve\to0}\sup_{\phi\in\sphere(\orbit,2\mu)}\ve\log\P\lb\sigma_\mu^{\ve,\phi}>T\rb\leq-\underline{V}.\ee
Define the closed set $F\subset C([-\delay,T],\R^d)$ by
	\be\label{eq:Fdef}F=\lcb x\in C([-\delay,T],\R^d):\x_t\in\domain^c\text{ for some }t\in[0,T]\rcb.\ee
By the uniform LDP upper bound shown in part 2 of Theorem \ref{thm:uniformLDP},
	\be\label{eq:supphiFeta}\limsup_{\ve\to0}\sup_{\phi\in \sphere(\orbit,2\mu)}\ve\log \P(\X^\ve\in F)\leq-\lim_{\eta\to0}\inf_{\phi\in\sphere(\orbit,2\mu)}I_T^\phi(F^\eta),\ee
where $F^\eta$ is defined as in \eqref{eq:Feta} for $\eta>0$. Next, we show that
	\be\label{eq:limetaIFetaV}\lim_{\eta\to0}\inf_{\phi\in\sphere(\orbit,2\mu)}I_T^\phi(F^\eta)\geq\underline{V}-2\alpha.\ee
Let $x\in F^{\eta^\dagger}$. By \eqref{eq:Feta}, there exists $y\in F$ such that $d_T(x,y)\leq\eta^\dagger$. By \eqref{eq:Fdef}, there exists $t\in[0,T]$ such that $y_t\in\domain^c$. Thus,
	$$d_0(\domain^c,\x_t)\leq d_0(\domain^c,y_t)+d_0(x_t,y_t)\leq d_T(x^\dagger,y)\leq\eta^\dagger.$$
It follows from Lemma \ref{lem:concat}, Lemma \ref{lem:IphiVetaalpha} and our choice of $\eta^\dagger>0$ that 
	$$\inf_{\phi\in\sphere(\orbit,\mu)}I_T^\phi(x)\geq\inf_{\phi\in\sphere(\orbit,\mu)}I_t^\phi(x|_{[-\delay,t]})\geq V_{2\eta^\dagger}-\alpha\geq\underline{V}-2\alpha.$$
Since the last display holds for all $x\in F^{\eta^\dagger}$, we see that \eqref{eq:limetaIFetaV} holds. Along with \eqref{eq:supphiFeta}, this implies
	\be\label{eq:XveFVdelta}\limsup_{\ve\to0}\sup_{\phi\in \sphere(\orbit,2\mu)}\ve\log \P(\X^{\ve,\phi}|_{[-\delay,T]}\in F)\leq-\underline{V}+2\alpha.\ee
By \eqref{eq:exitmu},
	$$\P\lb\X_{\sigma_\mu^{\ve,\phi}}^{\ve,\phi}\in\domain^c\rb\leq\P\lb\X^{\ve,\phi}|_{[-\delay,T]}\in F\rb+\P(\sigma_\mu^{\ve,\phi}>T).$$
It then follows from \eqref{eq:loglimit}, \eqref{eq:sigmaTV} and \eqref{eq:XveFVdelta} that \eqref{eq:phiS2muVbar} holds.
\end{proof}

\begin{lemma}\label{lem:T0}
For each $\mu>0$,
	\be\label{eq:Tzeroinfty}\lim_{S\to0}\limsup_{\ve\to0}\sup_{\phi\in\sphere(\orbit,\mu)}\ve\log\P\lb\sup_{0\leq s\leq S}d_0(\orbit,\X_s^{\ve,\phi})\geq2\mu\rb=-\infty.\ee
\end{lemma}

\begin{proof}
Fix $\mu>0$. For $S>0$ define the closed set
	\be\label{eq:FS}F_S=\lcb x\in C([-\delay,S],\R^d):\sup_{0\leq s\leq S}d_0(\orbit,\x_s)\geq2\mu\rcb.\ee
For $S>0$ and $\eta>0$, define $F_S^\eta$ as in \eqref{eq:Feta}, but with $S$, $F_S^\eta$ and $F_S$ in place of $T$, $F^\eta$ and $F$, respectively. It follows from \eqref{eq:Feta}, \eqref{eq:FS} and the triangle inequality that for all $S>0$ and $\eta>0$,
\begin{align}\label{eq:FSeta}
	F_S^\eta&\subset\lcb\x\in C([-\delay,S],\R^d):\sup_{0\leq s\leq S}d_0(\orbit,\x_s)\geq2\mu-\eta\rcb.
\end{align}
For $S>0$, by the uniform LDP upper bound shown in part 1 of Theorem \ref{thm:uniformLDP}, we have
	$$\limsup_{\ve\to0}\sup_{\phi\in \sphere(\orbit,\mu)}\ve\log\P\lb\sup_{0\leq s\leq S}d_0(\orbit,\X_s^{\ve,\phi})>2\mu\rb\leq-\lim_{\eta\to0}\inf_{\phi\in\ball(\orbit,\mu)}I_S^\phi(F_S^\eta).$$
By the definition of the rate function in \eqref{eq:rate} and \eqref{eq:Feta}, $I_S^\phi(F_S^\eta)$ is nondecreasing as $\eta\to0$. Thus, it suffices to show that for some $\eta>0$,
	\be\label{eq:limT0IFetaInf}\lim_{S\to0}\inf_{\phi\in\sphere(\orbit,\mu)}I_S^\phi(F_S^\eta)=\infty.\ee
Fix $\eta\in(0,\mu/4)$. By \eqref{eq:lipbound}, we can choose $M_\mu>0$ such that 
	\be\label{eq:Mb}|\f(\phi)|+|\g(\phi)|^2\leq M_\mu,\qquad\phi\in\overline{\ball(\orbit,2\mu)}.\ee 
Let
	\be\label{eq:Scond}0<S<\frac{\mu}{4M_\mu}.\ee 
Suppose $\x\in F_S^\eta$ is such that $\inf_{\phi\in\sphere(\orbit,\mu)}I_S^\phi(\x)<\infty$. Then \eqref{eq:rate} implies $\x_0\in\sphere(\orbit,\mu)$ and $U_S(\x)$ is nonempty. Let $t^\ast\in[0,\period)$ be such that 
	\be\label{eq:d0xsopsx0}d_0(\xsops_{t^\ast},\x_0)=d_0(\orbit,\x_0)=\mu.\ee 
Since $\x\in F_S^\eta$ and $s\to\x_s$ is a continuous function from $[0,S]$ to $\C$, it follows from \eqref{eq:d0xsopsx0} and \eqref{eq:FSeta} that there exists $t\in[0,S]$ such that $\x_s\in\ball(\orbit,2\mu)$ for all $s\in[0,t]$ and 
	$$\sup_{u\in[-\delay,0]}|\xsops(t^\ast+t+u)-x(t+u)|=d_0(\xsops_{t^\ast+t},\x_t)\geq d_0(\orbit,\x_t)\geq2\mu-\eta.$$ 
Let $t^\dagger\in[t-\delay,t]$, be such that
	\be\label{eq:xsopstdagger}|\xsops(t^\ast+t^\dagger)-\x(t^\dagger)|\geq2\mu-\eta.\ee
By \eqref{eq:d0xsopsx0}, \eqref{eq:xsopstdagger} and the fact that $\eta\in(0,\mu/4)$, we must have $t^\dagger>0$. Thus, we have the following inequalities, which are explained below:
\begin{align*}
	2\mu-\eta&\leq|\xsops(t^\ast)-\x(0)|+\int_0^t|\f(\xsops_{t^\ast+s})-\f(\x_s)|ds+\left|\int_0^t\g(\x_s)u(s)ds\right|\\
	&\leq\frac{3\mu}{2}+\sqrt{M_\mu S\int_0^S|u(s)|^2ds}.
\end{align*}
The first inequality follows from \eqref{eq:xsopstdagger}, \eqref{eq:dde} and \eqref{eq:xv}. The second inequality is due to \eqref{eq:d0xsopsx0}, \eqref{eq:Mb}, \eqref{eq:Scond} and the Cauchy-Schwarz inequality. Rearranging yields
	$$\frac{1}{2}\int_0^S|u(s)|^2ds>\frac{\mu^2}{32SM_\mu},$$
where we have used the fact that $\eta\in(0,\mu/4)$. Since this holds for all $x\in F_S^\eta$ satisfying $\inf_{\phi\in\ball(\orbit,\mu)}I_S^\phi(x)<\infty$ and all $u\in U_S(x)$, it follows from the definition of the rate function in \eqref{eq:rate} that
	$$I_S(F_S^\eta)\geq\frac{\mu^2}{32SM_\mu}.$$
Upon letting $S\to0$ we see that \eqref{eq:limT0IFetaInf} holds. This completes the proof of the lemma.
\end{proof}

\subsection{Exit time upper bound}

In this section we prove upper bounds on the exit time $\exit^{\ve,\phi}$ defined in \eqref{eq:exitvedelta}.

\begin{lemma}\label{lem:exitupper}
Let $\phi\in\domain$. Then
\begin{equation}\label{eq:exitupper1}
	\limsup_{\ve\to0}\ve\log\E\lsb\exit^{\ve,\phi}\rsb\leq \overline{V}
\end{equation}
and for all $\alpha>0$,
	\be\label{eq:exitupper2}\lim_{\ve\to0}\P\lb\exit^{\ve,\phi}< e^{(\overline{V}+\alpha)/\ve}\rb=1.\ee
\end{lemma}

\begin{proof}
Fix $\phi\in\domain$. Suppose that for each $\alpha>0$ there exists $\ve_0>0$ such that
\begin{equation}\label{eq:exitupper3}
	\ve\log\E\lsb\exit^{\ve,\phi}\rsb<\overline{V}+\frac{\alpha}{2},\qquad\ve\in(0,\ve_0).
\end{equation}
Then \eqref{eq:exitupper1} follows by first letting $\ve\to0$ and then $\alpha\to0$. Additionally, \eqref{eq:exitupper3} along with Chebyshev's inequality implies that for each $\alpha>0$ there exists $\ve_0>0$ such that
	$$\P\lb\exit^{\ve,\phi}\geq e^{(\overline{V}+\alpha)/\ve}\rb<e^{-\alpha/2\ve}\quad\text{for all }\ve\in(0,\ve_0).$$
Sending $\ve\to0$, we obtain \eqref{eq:exitupper2}. Therefore, we are left to prove that for each $\alpha>0$ there exists $\ve_0>0$ such that \eqref{eq:exitupper3} holds.

Fix $\alpha>0$. We first provide a lower bound on the probability that a solution of the SDDE exits $\domain$ in a finite time interval when starting near the orbit. By Lemma \ref{lem:NmuNh}, there are constants $\mu,h,T\in(0,\infty)$ such that for every $\psi\in\overline{\ball(\orbit,\mu)}$, there exists $T(\psi)\leq T$ and $\hat{x}^\psi\in C([-\delay,T(\psi)],\R^d)$ satisfying
\begin{equation}\label{eq:exitpath}
	\hat\x_0^\psi=\psi,\qquad\hat\x_{T(\psi)}^\psi\not\in\ball(\domain,h),\qquad I_{T(\psi)}(\hat\x^\psi)<\overline{V}+\frac{\alpha}{8}.
\end{equation}
For each $\psi\in\overline{\ball(\orbit,\mu)}$, define $\tilde x^\psi\in C([-\delay,T],\R^d)$ by
	\be\label{eq:tildexpsi}\tilde x^\psi(t)=\begin{cases}\hat x^\psi(t),&t\in[-\delay,T(\psi)),\\ x^{\hat x_{T(\psi)}^\psi}(t-T(\psi)),&t\in[T(\psi),T],\end{cases}\ee
where we recall that $x^{\hat x_{T(\psi)}^\psi}$ denotes the solution of the DDE with initial condition $\hat\x_{T(\psi)}^\psi$. By \eqref{eq:tildexpsi}, \eqref{eq:exitpath}, Lemma \ref{lem:concat}, and Remark \ref{rmk:ratexphi}, it follows that $\tilde x^\psi$ satisfies
	\be\label{eq:exitpath1}\tilde\x_0^\psi=\psi,\qquad\tilde\x_{T(\psi)}^\psi\not\in\ball(\domain,h),\qquad I_T(\tilde\x^\psi)<\overline{V}+\frac{\alpha}{8}.\ee
Define the open set $G\subset C([-\delay,T],\R^d)$ to be the union of open balls defined by
	\be\label{eq:Gunion}G=\bigcup_{\psi\in\overline{\ball(\orbit,\mu)}}\lcb x\in C([-\delay,T],\R^d):d_T(\tilde\x^\psi,x)<h\rcb.\ee
For $\eta>0$ define $G_\eta$ as in \eqref{eq:Geta}. It follows from \eqref{eq:Geta} and \eqref{eq:Gunion} that $\tilde x^\psi\in G_\eta$ for all $\psi\in\overline{\ball(\orbit,\mu)}$ and $\eta\in(0,h)$. Thus, by the definition of the rate function in \eqref{eq:rate} and \eqref{eq:exitpath1}, we have
	\be\label{eq:limetaIGeta}\lim_{\eta\to0}\sup_{\psi\in\overline{\ball(\orbit,\mu)}}I_T^\psi(G_\eta)\leq\sup_{\psi\in\overline{\ball(\orbit,\mu)}}I_T^\psi(\tilde x^\psi)<\overline{V}+\frac{\alpha}{8}.\ee
Suppose $\ve>0$ and $\phi\in\domain$. By \eqref{eq:Gunion}, \eqref{eq:exitpath1} and the definition of $\exit^{\ve,\phi}$ in \eqref{eq:exitvedelta}, we see that
\begin{align*}
	\lcb\X^{\ve,\phi}|_{[-\delay,T]}\in G\rcb&\subset\lcb d_T(\tilde{x}^\psi,\X^{\ve,\phi}|_{[-\delay,T]})<h\text{ for some }\psi\in\overline{\ball(\orbit,\mu)}\rcb\\
	&\subset\lcb\X_{T(\psi)}^{\ve,\phi}\not\in\domain\text{ for some }\psi\in\overline{\ball(\orbit,\mu)}\rcb\\
	&\subset\lcb\rho^{\ve,\phi}<T(\psi)\leq T\rcb.
\end{align*}
Therefore, by the uniform LDP lower bound shown in part 2 of Theorem \ref{thm:uniformLDP} and \eqref{eq:limetaIGeta}, there exists $\ve_0>0$ such that for all $\ve\in(0,\ve_0)$,
\begin{align}\label{eq:psiNmuexit}
	\inf_{\phi\in\overline{\ball(\orbit,\mu)}}\P\lb\exit^{\ve,\phi}<T\rb&\geq\inf_{\phi\in\overline{\ball(\orbit,\mu)}}\P\left(\X^{\ve,\phi}|_{[-\delay,T]}\in G\right)\\ \notag
	&\geq\exp\lcb-\ve^{-1}\lb\lim_{\eta\to0}\sup_{\phi\in\overline{\ball(\orbit,\mu)}} I_T^\phi(G_\eta)+\frac{\alpha}{8}\rb\rcb\\ \notag
	&>\exp\lcb-\ve^{-1}\lb \overline{V}+\frac{\alpha}{4}\rb\rcb.
\end{align}

Next we provided a lower bound on the probability that, given $\phi\in\domain$, the solution $\X^{\ve,\phi}$ will reach $\overline{\ball(\orbit,\mu)}$ within a finite time interval. Since $\domain$ is uniformly attracted to $\orbit$ (see Definition \ref{def:attracted}), there exists $S>0$ such that given any $\phi\in\domain$, we have
	\be\label{eq:rho0tgeqS}\x_t^\phi\in\ball(\orbit,\mu/2)\quad\text{for all }t\geq S.\ee
By Remark \ref{rmk:distribution}, 
	\be\label{eq:Xvephimu2}\lim_{\ve\to0}\sup_{\phi\in\domain}\P\lb\norm{\X^{\ve,\phi}-\x^\phi}_{[-\delay,S]}\geq\frac{\mu}{2}\rb=0.\ee
Let 
	\be\label{eq:thetavephi}\theta^{\ve,\phi}=\inf\lcb t\geq0:\X_t^{\ve,\phi}\in\overline{\ball(\orbit,\mu)}\rcb.\ee
Due to \eqref{eq:rho0tgeqS}, we see that $\norm{\X^{\ve,\phi}-\x^\phi}_{[-\delay,S]}<\mu/2$ implies $\X_S^{\ve,\phi}\in\ball(\orbit,\mu)$ and so $\theta^{\ve,\phi}<S$. Thus, by \eqref{eq:Xvephimu2}, there exists $\ve_1\in(0,\ve_0)$ satisfying
	\be\label{eq:STve0}\ve_1<\frac{\alpha}{4\log(2(S+T))},\ee
such that
	\be\label{eq:thetaveS}\P(\theta^{\ve,\phi}<S)\geq\frac{1}{2},\qquad\phi\in\domain,\;\ve\in(0,\ve_1).\ee

We are now ready to complete the proof. By the strong Markov property, \eqref{eq:thetavephi}, \eqref{eq:thetaveS} and \eqref{eq:psiNmuexit}, for all $\phi\in\domain$ and $\ve\in(0,\ve_1)$, we have
\begin{align*}
	\P(\exit^{\ve,\phi}<S+T)&\geq\P(\theta^{\ve,\phi}<S,\exit^{\ve,\phi}-\theta^{\ve,\phi}<T)\\
	&\geq\P(\theta^{\ve,\phi}<S)\cdot\inf_{\psi\in\overline{\ball(\orbit,\mu)}}\P(\exit^{\ve,\psi}<T)\\
	&\geq \frac{1}{2}\exp\lcb-\ve^{-1}\lb \overline{V}+\frac{\alpha}{4}\rb\rcb.
\end{align*}
Again invoking the strong Markov property, we have, for all $\phi\in\domain$ and $\ve\in(0,\ve_1)$,
\begin{align*}
	E[\exit^{\ve,\phi}]&\leq(S+T)\sum_{n=0}^\infty \P(\exit^{\ve,\phi}\geq n(S+T))\\
	&\leq(S+T)\sum_{n=0}^\infty\lsb\sup_{\psi\in\domain}\P(\exit^{\ve,\psi}\geq S+T)\rsb^n\\
	&\leq(S+T)\sum_{n=0}^\infty\lsb1-\inf_{\psi\in\domain}\P\left(\exit^{\ve,\psi}< S+T\right)\rsb^n\\
	&\leq(S+T)\left(\inf_{\psi\in\domain}\P(\exit^{\ve,\psi}< S+T)\right)^{-1}\\
	&\leq 2(S+T)\exp\lcb\ve^{-1}\lb \overline{V}+\frac{\alpha}{4}\rb\rcb.
\end{align*}
Along with \eqref{eq:STve0} this implies that that \eqref{eq:exitupper3} holds, thus completing the proof of the lemma.
\end{proof}

\subsection{Exit time lower bound}

In this section we prove lower bounds on the exit time $\exit^{\ve,\phi}$ defined in \eqref{eq:exitvedelta}.

\begin{lemma}\label{lem:exitlower}
Let $\phi\in\domain^s$. Then
	\be\label{eq:exitlower1}\liminf_{\ve\to0}\ve\log \E\lsb\exit^{\ve,\phi}\rsb\geq \underline{V},\ee
and for all $\alpha>0$,
	\be\label{eq:exitlower2}\lim_{\ve\to0}\P\lb\exit^{\ve,\phi}> e^{(\underline{V}-\alpha)/\ve}\rb=1.\ee
\end{lemma}

\begin{proof}
Fix $\phi\in\domain^s$. Suppose \eqref{eq:exitlower2} holds for all $\alpha>0$. Then by Chebyshev's inequality, for $\alpha>0$,
	$$\liminf_{\ve\to0}\ve\log E[\exit^{\ve,\phi}]\geq \underline{V}-\alpha.$$
Since $\alpha>0$ was arbitrary, this implies \eqref{eq:exitlower1} holds. Thus, we are left to prove \eqref{eq:exitlower2} holds for all $\alpha>0$.

Let $\alpha>0$. By Lemma \ref{lem:phiS2muVbar}, we can choose $\mu>0$ and $\ve_0>0$ such that
	\be\label{eq:Smuexitmu}\sup_{\psi\in \sphere(\orbit,2\mu)}\P\lb\X_{\sigma_\mu^{\ve,\psi}}^{\ve,\psi}\in\domain^c\rb<e^{-(\underline{V}-\alpha/2)/\ve},\qquad\ve\in(0,\ve_0),\ee
where $\sigma_\mu^{\ve,\phi}$ is defined as in \eqref{eq:exitmu}. Let $0\leq\xi_1^{\ve,\phi}<\upsilon_1^{\ve,\phi}<\xi_2^{\ve,\phi}<\cdots$ be the nested sequence of increasing $\{\F_t\}$-stopping times defined as follows: set 
	\be\label{eq:xi1}\xi_1^{\ve,\phi}=\sigma_\mu^{\ve,\phi}=\inf\lcb t\geq0:\X_t^{\ve,\phi}\in\sphere(\orbit,\mu)\cup\domain^c\rcb,\ee
and for $n\geq1$ such that $\X_{\xi_n}^{\ve,\phi}\in\sphere(\orbit,\mu)$, recursively define
\begin{align}\label{eq:upsilonn}
	\upsilon_n^{\ve,\phi}&=\inf\lcb t>\xi_n^{\ve,\phi}:\X_t^{\ve,\phi}\in \sphere(\orbit,2\mu)\rcb,\\ \label{eq:xin}
	\xi_{n+1}^{\ve,\phi}&=\inf\lcb t>\upsilon_n^{\ve,\phi}:\X_t^{\ve,\phi}\in \sphere(\orbit,\mu)\cup\domain^c\rcb.
\end{align}
If $\X_{\xi_n}^{\ve,\phi}\in\domain^c$ for some $n\geq1$, then set $\upsilon_n^{\ve,\phi}=\infty$ and end the sequence of stopping times. Observe that, by \eqref{eq:exitvedelta} and \eqref{eq:xi1}--\eqref{eq:xin}, $\exit^{\ve,\phi}<\infty$ implies $\exit^{\ve,\phi}=\xi_n^{\ve,\phi}$ for some $n\geq1$. Therefore, given $k\geq1$ and $T>0$, we have
	\be\label{eq:Pphisigmeve}\P(\exit^{\ve,\phi}<kT)=\P\lb\bigcup_{n=1}^\infty\{\exit^{\ve,\phi}=\xi_n^{\ve,\phi}<kT\}\rb.\ee
By \eqref{eq:exitvedelta}, \eqref{eq:xi1} and Lemma \ref{lem:deltaDexitmu}, we can choose $\ve_1\in(0,\ve_0)$ such that
	\be\label{eq:exitvesigma0}\P\lb\exit^{\ve,\phi}=\xi_1^{\ve,\phi}\rb<e^{-(\underline{V}-\alpha/2)/\ve},\qquad\ve\in(0,\ve_1).\ee
Given $\ve\in(0,\ve_1)$, by \eqref{eq:exitvedelta}, \eqref{eq:xin}, \eqref{eq:upsilonn}, the strong Markov property, \eqref{eq:exitmu} and \eqref{eq:Smuexitmu}, for each $n\geq1$,
\begin{align}\label{eq:exitvesigman}
	\P\lb\exit^{\ve,\phi}=\xi_{n+1}^{\ve,\phi}\rb&=\P\lb\upsilon_n^{\ve,\phi}<\infty,\exit^{\ve,\phi}=\xi_{n+1}^{\ve,\phi}\rb\\ \notag
	&\leq\sup_{\psi\in\sphere(\orbit,2\mu)}P\lb\X_{\sigma_\mu^{\ve,\psi}}^{\ve,\psi}\in\domain^c\rb\\ \notag
	&<e^{-(\underline{V}-\alpha/2)/\ve}.
\end{align}
By \eqref{eq:upsilonn}, \eqref{eq:xin}, the strong Markov property and Lemma \ref{lem:T0}, we can choose $S>0$ and $\ve_2\in(0,\ve_1)$ such that for all $\ve\in(0,\ve_2)$ and $n\geq1$,
\begin{align}\label{eq:thetansigman}
	\P(\upsilon_n^{\ve,\phi}-\xi_n^{\ve,\phi}\leq S)&\leq\sup_{\psi\in\sphere(\orbit,\mu)}P\lb\sup_{0\leq t\leq S}d_0(\orbit,\X_t^{\ve,\psi})\geq2\mu\rb\\ \notag
	&\leq e^{-(\underline{V}-\alpha/2)/\ve}.
\end{align}
By \eqref{eq:Pphisigmeve}--\eqref{eq:thetansigman}, for $k\geq1$ and all $\ve\in(0,\ve_2)$,
\begin{align}\label{eq:exitkT0}
	\P(\exit^{\ve,\phi}<kS)&\leq\sum_{n=1}^k \P(\exit^{\ve,\phi}=\xi_n^{\ve,\phi})+P\lb\min_{1\leq n\leq k}(\upsilon_n^{\ve,\phi}-\xi_n^{\ve,\phi})\leq S\rb\\ \notag
	&\leq ke^{-(\underline{V}-\alpha/2)/\ve}+\sum_{n=1}^kP\lb\upsilon_n^{\ve,\phi}-\xi_n^{\ve,\phi}\leq S\rb\\ \notag
	&\leq2ke^{-(\underline{V}-\alpha/2)/\ve},
\end{align}
where the first inequality uses the fact that 
	$$\bigcup_{n=k+1}^\infty\{\xi_n^{\ve,\phi}< kS\}\subset\lcb\min_{1\leq n\leq k}(\upsilon_n^{\ve,\phi}-\xi_n^{\ve,\phi})\leq S\rcb.$$
From the estimate \eqref{eq:exitkT0}, with 
	$$k=\max\lcb\ell\geq1:\ell\leq S^{-1}e^{(\underline{V}-\alpha)/\ve}\rcb+1,$$ 
we have for all $\ve\in(0,\ve_2)$,
	$$\P\lb\exit^{\ve,\phi}\leq e^{(\underline{V}-\alpha)/\ve}\rb\leq \P(\exit^{\ve,\phi}<kS)\leq 2S^{-1}e^{-\alpha/2\ve}+2e^{-(\underline{V}-\alpha/2)/\ve}.$$
Letting $\ve\to0$ yields \eqref{eq:exitlower2}, which completes the proof of the lemma.
\end{proof}

\subsection{Proof of Lemma \ref{lem:Vupperlower}}\label{sec:Vupperlower}

\begin{proof}[Proof of Lemma \ref{lem:Vupperlower}]
It suffices to show that $\underline{V}\geq\overline{V}$. Let $\alpha>0$ and choose
	\be\label{eq:eta}0<\eta<\min\lb\frac{\alpha}{4M_a(\lip_1+1)\delta},\frac{\delta}{2}\rb.\ee
By \eqref{eq:Vlower} and \eqref{eq:Veta}, $V_\eta$ is nondecreasing as $\eta\to0$ and we can choose $\psi\in\C$ such that $d_0(\domain^c,\psi)<\eta$ and $V(\psi)<V_\eta+\alpha\leq\underline{V}+\alpha$. Then, according to \eqref{eq:quasipotential}, there exists $T^\dagger>0$ and $x^\dagger\in C([-\delay,T^\dagger],\R^d)$ such that $\x_0^\dagger=\phi^\ast$, $\x_{T^\dagger}^\dagger=\psi$ and 
	\be\label{eq:IT1xdagger}I_{T^\dagger}(\x^\dagger)<V(\psi)+\alpha<\underline{V}+2\alpha.\ee
Note that 
	$$\domain^c=\lcb\phi\in\C:|\phi(s)-\nu^\ast|\geq\delta\text{ for some }s\in[-\delay,0]\rcb.$$
Since $d_0(\domain^c,\x_0^\dagger)=\delta$, $d_0(\domain^c,\x_{T^\dagger}^\dagger)<\eta$ and the function $t\to d_0(\domain^c,\x_t^\dagger)$ from $[0,T^\dagger]$ to $\R_+$ is continuous, there exists $S\in[0,T^\dagger]$ such that $|\x^\dagger(t)-\nu^\ast|<\delta-\eta$ for all $t\in[-\delay,S)$ and
	\be\label{eq:xdaggerSdeltaeta}|\x^\dagger(S)-\nu^\ast|=\delta-\eta.\ee
By Lemma \ref{lem:concat} and \eqref{eq:IT1xdagger},
	\be\label{eq:ISxdagger} I_S(\x^\dagger|_{[-\delay,S]})\leq I_{T^\dagger}(\x^\dagger)<\underline{V}+2\alpha.\ee 
Define $T=S+\frac{2\eta}{\delta-\eta}$ and $\x\in C([-\delay,T],\R^d)$ by
	\be\label{eq:xdaggerST}\x(t)=\begin{cases}\x^\dagger(t),&t\in[-\delay,S],\\ \x^\dagger(S)+(t-S)(\x^\dagger(S)-\nu^\ast),&t\in(S,T].\end{cases}\ee
Then \eqref{eq:xdaggerST}, the definition of $T$ and \eqref{eq:xdaggerSdeltaeta} imply $|x(T)-\nu^\ast|=\delta+\eta$, so $\x_T\not\in\overline{\domain}$. By \eqref{eq:Vbar}, the fact that $x_T\not\in\overline{\domain}$, Lemma \ref{lem:concat} and \eqref{eq:ISxdagger}, we have
\begin{align}\label{eq:VoverVunderITS2alpha}
	\overline{V}\leq I_T\lb\x\rb=I_S\lb\x^\dagger|_{[-\delay,S]}\rb+I_{T-S}\lb\x^S\rb<\underline{V}+I_{T-S}\lb\x^S\rb+2\alpha,
\end{align}
where $\x^S\in C([-\delay,T-S],\R^d)$ is defined by $\x^S(t)=\x(S+t)$ for $t\in[-\delay,T-S]$. We have the following inequalities, which we explain below:
\begin{align*}
	I_{T-S}(\x^S)&\leq\frac{1}{2}\int_S^T|(a(\x_s))^{-1}||\f(\x_s)-\f(\phi^\ast)-\dot{\x}(s)|^2ds\\
	&\leq M_a(T-S)\lb2\lip_1\sup_{s\in[S-\delay,T]}|\x(s)-\nu^\ast|^2+2|\x^\dagger(S)-\nu^\ast|^2\rb\\
	&\leq 4M_a(\lip_1+1)(\delta-\eta)\eta\\
	&<\alpha.
\end{align*}
The first inequality is due to Lemma \ref{lem:rateAC}, the definition of $\x^S$, and the fact that $\f(\phi^\ast)=0$ since $\phi^\ast$ is an equilibrium point of the DDE. The second inequality follows from \eqref{eq:Mabound}, the continuity of $\f$ (Assumption \ref{ass:lip}), and the fact that $x$ is differentiable on $(S,T)$ with derivative equal to $\frac{dx(t)}{dt}=x^\dagger(S)-\nu^\ast$ for all $t\in(S,T)$ by \eqref{eq:xdaggerST}. The third inequality is due the definition of $T$, the definition of $x$ in \eqref{eq:xdaggerST}, and \eqref{eq:xdaggerSdeltaeta}. The last inequality is due to \eqref{eq:eta}. By \eqref{eq:VoverVunderITS2alpha} and the last display, $\overline{V}<\underline{V}+3\alpha$. Since $\alpha>0$ was arbitrary, this completes the proof that $\underline{V}\geq\overline{V}$.
\end{proof}

{\bf Acknowledgments}. The author is grateful to Ruth Williams for suggesting this problem. The author thanks Amarjit Budhiraja, Paul Dupuis and Mickey Salins for conversations concerning uniform LDPs over bounded sets and for sharing a preprint of their recent work \cite{Budhiraja2017} on this topic. The author also thanks Mickey Salins for pointing out an error in an earlier statement of Theorem \ref{thm:uniformLDP}, and Kavita Ramanan for several helpful comments on a preliminary draft of this work.

\bibliographystyle{plain}
\bibliography{sdde}

\begin{thebibliography}{10}

\bibitem{Anderson2007}
D.~F. Anderson.
\newblock {A modified next reaction method for simulating chemical systems with
  time dependent propensities and delays.}
\newblock {\em J. Chemical Phys.}, 127(21):214107, dec 2007.

\bibitem{Azencott2016}
R.~Azencott, B.~Geiger, and W.~Ott.
\newblock Large deviations for {G}aussian diffusions with delay.
\newblock arXiv preprint arXiv:1610.08769v1, 2016.

\bibitem{Boue1998}
M.~Bou{\'{e}} and P.~Dupuis.
\newblock {A variational representation for certain functions of Brownian
  motion}.
\newblock {\em Ann. Probab.}, 26(4):1641--1659, 1998.

\bibitem{Bratsun2005}
D.~Bratsun, D.~Volfson, L.~S. Tsimring, and J.~Hasty.
\newblock {Delay-induced stochastic oscillations in gene regulation.}
\newblock {\em Proc. Natl. Acad. Sci. U. S. A.}, 102(41):14593--14598, 2005.

\bibitem{Budhiraja2000}
A.~Budhiraja and P.~Dupuis.
\newblock A variational representation for positive functionals of infinite
  dimensional {B}rownian motions.
\newblock {\em Probab. Math. Statist.}, 20:39--61, 2000.

\bibitem{Budhiraja2008}
A.~Budhiraja, P.~Dupuis, and V.~Maroulas.
\newblock Large deviations for infinite dimensional stochastic dynamical
  systems.
\newblock {\em Ann.\ Probab.}, 36:1390--1420, 2008.

\bibitem{Budhiraja2011}
A.~Budhiraja, P.~Dupuis, and V.~Maroulas.
\newblock Variational representations for continuous time processes.
\newblock {\em Ann. Inst. H. Poincar{\'{e}} Probab. Stat.}, 47(3):725--747,
  2011.

\bibitem{Budhiraja2017}
A.~Budhiraja, P.~Dupuis, and M.~Salins.
\newblock Uniform large deviations principle for {B}anach space valued
  stochastic differential equations.
\newblock Preprint, 2017.

\bibitem{Cerrai2004}
S.~Cerrai and M.~R{\"o}ckner.
\newblock Large deviations for stochastic reaction-diffusion systems with
  multiplicative noise and non-{Lipschitz} reaction term.
\newblock {\em Ann. Probab.}, 32:1100--1139, 2004.

\bibitem{Chenal1997}
F.~Chenal and M.~Millet.
\newblock Uniform large deviations for parabolic {SPDEs} and their
  applications.
\newblock {\em Stoch. Process. Appl.}, 72(2):161--186, 1997.

\bibitem{Chow1988}
S.-N. Chow and H.-O. Walther.
\newblock Characteristic multipliers and stability of symmetric periodic
  solutions of $\dot{x}(t)=g(x(t-1))$.
\newblock {\em Trans. Amer. Math. Soc.}, 307:127--142, 1988.

\bibitem{Coombes2009}
S.~Coombes and C.~Laing.
\newblock Delays in activity-based neural networks.
\newblock {\em Phil.\ Trans.\ R.\ Soc.\ A}, 367:1117–--1129, 2009.

\bibitem{daPrato1992}
G.~da~Prato and J.~Zabczyk.
\newblock {\em Stochastic Equations in Infinite Dimensions}.
\newblock Cambridge University Press, Great Britain, 1992.

\bibitem{Day1990a}
M.~Day.
\newblock Large deviations results for the exit problem with characteristic
  boundary.
\newblock {\em J.\ Math Anal.\ and Appl.}, 147:134--153, 1990.

\bibitem{Day1990b}
M.~Day.
\newblock Some phenomena of the characteristic boundary exit problem.
\newblock In R.~Pinsky, editor, {\em Diffusion processes and related problems
  in analysis}, pages 55--72. Birkh{\"a}user, Basel, Switzerland, 1990.

\bibitem{Day1992}
M.~Day.
\newblock Conditional exits for small noise diffusions with characteristic
  boundary.
\newblock {\em Ann.\ Probab.}, 20:1385--–1419, 1992.

\bibitem{Dembo1998}
A.~Dembo and O.~Zeitouni.
\newblock {\em {Large Deviations Techniques and Applications}}.
\newblock Springer, Berlin Heidelberg, 1998.

\bibitem{Driver1965}
R.~D. Driver.
\newblock Existence and continuous dependence of solutions of a neutral
  functional differential equation.
\newblock {\em Arch. Rat. Mech. Anal.}, 19:149--166, 1965.

\bibitem{Dupuis1997}
P.~Dupuis and R.~S. Ellis.
\newblock {\em {A Weak Convergence Approach to Large Deviations}}.
\newblock Wiley, New York, 1997.

\bibitem{Dupuis1986}
P.~Dupuis and H.~J. Kushner.
\newblock Large deviations for systems with small noise effects, and
  applications to stochastic systems theory.
\newblock {\em SIAM J.\ Cont.\ Opt.}, 24:979–1008, 1986.

\bibitem{Dupuis1987}
P.~Dupuis and H.~J. Kushner.
\newblock Stochastic systems with small noise, analysis and simulation; a phase
  locked loop example.
\newblock {\em SIAM J.\ Appl.\ Math.}, 47:643--661, 1987.

\bibitem{Eizenberg1984}
A.~Eizenberg.
\newblock The exit distribution for small random perturbations of dynamical
  systems with a repulsive type stationary point.
\newblock {\em Stochastics}, 12:251--275, 1984.

\bibitem{Eizenberg1987}
A.~Eizenberg and Y.~Kifer.
\newblock The asymptotic behavior of the principal eigenvalue in a singular
  perturbation problem with invariant boundaries.
\newblock {\em Prob.\ Th.\ Rel.\ Fields}, 76:439--476, 1987.

\bibitem{Freidlin1984}
M.~I. Freidlin and A.~D. Wentzell.
\newblock {\em Random Perturbations of Dynamical Systems}.
\newblock Springer-Verlag, New York, 1984.

\bibitem{Freidlin2012}
M.~I. Freidlin and A.~D. Wentzell.
\newblock {\em {Random Perturbations of Dynamical Systems}}.
\newblock Springer, New York, 2012.

\bibitem{Hale1965}
J.~K. Hale.
\newblock Sufficient conditions for stability and instability of automonous
  functional differential equations.
\newblock {\em J. Differential Equations}, 1:452--482, 1965.

\bibitem{Hale1993}
J.~K. Hale and S.~M. Verduyn~Lunel.
\newblock {\em Introduction to Functional Differential Equations}.
\newblock Springer-Verlag, New York, 1993.

\bibitem{Huang1989}
W.~Huang.
\newblock Generalization of {L}iapunov's theorem in a linear delay system.
\newblock {\em J. Math. Anal. Appl.}, 142:83--94, 1989.

\bibitem{Ivanov1999}
A.F. Ivanov and J.~Losson.
\newblock Stable rapidly oscillating solutions in delay differential equations
  with negative feedback.
\newblock {\em Differ.\ Integral Equ.}, 12:811–--832, 1999.

\bibitem{Kaplan1975}
J.~L. Kaplan and J.~A. Yorke.
\newblock On the stability of a periodic solution of a differential-delay
  equation.
\newblock {\em SIAM J. Math. Anal.}, 6:268--282, 1975.

\bibitem{Kaplan1977}
J.~L. Kaplan and J.~A. Yorke.
\newblock On the nonlinear differential delay equation $x'(t)=-f(x(t),x(t-1))$.
\newblock {\em J.\ Differential Equations}, 23:293--314, 1977.

\bibitem{Kifer1981}
Y.~Kifer.
\newblock The exit problem for small random perturbations of dynamical systems
  with a hyperbolic fixed point.
\newblock {\em Isarel J.\ Math.}, 40:74--96, 1981.

\bibitem{Kolmogorov1957}
A.~N. Kolmogorov and S.~F. Folmin.
\newblock {\em {Elements of the Theory of Functions and Functional Analysis}}.
\newblock Graylock, Rochester, 1957.

\bibitem{Langevin1991}
R.~Langevin, W.~M. Oliva, and J.~C.~F. de~Oliveira.
\newblock {Retarded functional differential equations with white noise
  perturbations}.
\newblock {\em Ann. l'Institut Henri Poincar{\'{e}}, Sect. A}, 55(2):671--687,
  1991.

\bibitem{Mallet-Paret2011}
J.~Mallet-Paret and N.~Nussbaum.
\newblock Superstability and rigorous asymptotics in singularly perturbed
  state-dependent delay-differential equations.
\newblock {\em J.\ Differential Equations}, 250:4037–--4084, 2011.

\bibitem{Mallet-Paret1996b}
J.~Mallet-Paret and G.~Sell.
\newblock The {P}oincar{\'e}-{B}endixson theorem for monotone cyclic feedback
  systems with delay.
\newblock {\em J.\ Differential Equations}, 125:441--489, 1996.

\bibitem{Mallet-Paret1996a}
J.~Mallet-Paret and G.~Sell.
\newblock Systems of differential delay equations: {F}loquet multipliers and
  discrete {L}yapunov functions.
\newblock {\em J.\ Differential Equations}, 125:385--440, 1996.

\bibitem{Mather2009}
W.~H. Mather, M.~R. Bennett, J.~Hasty, and L.~S. Tsimring.
\newblock {Delay-induced degrade-and-fire oscillations in small genetic
  circuits.}
\newblock {\em Phys. Rev. Lett.}, 102(6), feb 2009.

\bibitem{Mo2013}
C.~Mo and J.~Luo.
\newblock Large deviations for stochastic differential delay equations.
\newblock {\em Nonlinear Anal.}, 80:202--210, 2013.

\bibitem{Mohammed1984}
S.-E.~A. Mohammed.
\newblock {\em {Stochastic Functional Differential Equations}}.
\newblock Pitman, Boston, 1984.

\bibitem{Mohammed2006}
S.-E.~A. Mohammed and T.~Zhang.
\newblock {Large deviations for stochastic systems with memory}.
\newblock {\em Discret. Contin. Dyn. Syst. B}, 6(4):881--893, 2006.

\bibitem{PAG02}
F.~Paganini and Z.~Wang.
\newblock Global stability with time-delay in network congestion control.
\newblock In {\em Proceedings of the 41st IEEE Conference on Decision and
  Control, 2002}, pages 3632--3637, 2002.

\bibitem{PAP04a}
A.~Papachristodoulou.
\newblock Global stability analysis of {TCP/AQM} protocol for arbitrary
  networks with delay.
\newblock In {\em 43rd IEEE Conference on Decision and Control, 2004. CDC.},
  pages 1029--1034, 2004.

\bibitem{PAP04b}
A.~Papachristodoulou, J.~C. Doyle, and S.~H. Low.
\newblock Analysis of nonlinear delay differentiable equation models of
  {TCP/AQM} protocols using sums of squares.
\newblock In {\em 43rd IEEE Conference on Decision and Control, 2004. CDC.},
  pages 4684--4689, 2004.

\bibitem{PEE07}
M.~Peet and S.~Lall.
\newblock Global stability analysis of a nonlinear model of {I}nternet
  congestion control with delay.
\newblock {\em IEEE Trans. Automat. Control}, 52:553--559, 2007.

\bibitem{Razumikhin1956}
B.~S. Razumikhin.
\newblock On the stability of systems with a delay.
\newblock {\em Prikl. Mat. Meh.}, 20:500--512, 1956.

\bibitem{Razumikhin1960}
B.~S. Razumikhin.
\newblock Application of {L}iapunov's method to problems in the stability of
  systems with a delay.
\newblock {\em Automat. i Telemeh.}, 21:740--749, 1960.

\bibitem{Revuz1999}
D.~Revuz and M.~Mor.
\newblock {\em Continuous Martingales and Brownian Motion}.
\newblock Springer-Verlag, Berlin, 3rd edition, 1999.

\bibitem{Rogers2000a}
L.~C.~G. Rogers and D.~Williams.
\newblock {\em Diffusions, Markov Processes and Martingales: Volume 2, It{\^o}
  Calculus}.
\newblock Cambridge University Press, Cambridge, UK, second edition, 2000.

\bibitem{Roxin2005}
A.~Roxin, N.~Brunel, and D.~Hansel.
\newblock Role of delays in shaping spatiotemporal dynamics of neuronal
  activity in large networks.
\newblock {\em Phys.\ Rev.\ Lett.}, 94:238103, 2005.

\bibitem{Roxin2011}
A.~Roxin and E.~Montbri{\'o}.
\newblock How effective delays shape oscillatory dynamics in neuronal networks.
\newblock {\em Physica D: Nonlinear Phenomena}, 240:323--345, 2011.

\bibitem{Smith2011}
H.~Smith.
\newblock {\em An Introduction to Delay Differential Equations with
  Applications to the Life Sciences}.
\newblock Springer, New York, 2011.

\bibitem{Sowers1992}
R.~Sowers.
\newblock Large deviations for a reaction-diffusion equation with
  non-{G}aussian perturbations.
\newblock {\em Ann.\ Probab.}, 20(1):504--537, 1992.

\bibitem{Stoffer2008}
D.~Stoffer.
\newblock Delay equations with rapidly oscillating stable periodic orbits.
\newblock {\em J.\ Dynam. Differential Equations}, 20:201--–238, 2008.

\bibitem{Stoffer2011}
D.~Stoffer.
\newblock Two results on stable rapidly oscillating solutions of delay
  differential equations.
\newblock {\em Dyn.\ Syst.}, 1:169–--188, 2011.

\bibitem{Varadhan1966}
S.~R.~S. Varadhan.
\newblock Asymptotic probabilities and differential equations.
\newblock {\em Comm. Pure. Appl. Math.}, 19:261--286, 1966.

\bibitem{Ventcel1970}
A.~D. Ventcel and M.~I. Freidlin.
\newblock {On small random perturbations of dynamical systems}.
\newblock {\em Russ. Math. Surv.}, 25:1--55, 1970.

\bibitem{Ventcel1972}
A.~D. Ventcel and M.~I. Freidlin.
\newblock Some problems concerning stability under small random perturbations.
\newblock {\em Theory Probab. Appl.}, 17:269---283, 1972.

\bibitem{Walther1995}
H.-O. Walther.
\newblock {\em The Two-Dimensional Attractor of $x'(t)=−µx(t)+f(x(t-1))$}.
\newblock American Mathematical Society, Providence, RI, 1995.

\bibitem{Walther2001b}
H.-O. Walther.
\newblock Contracting return maps for monotone delayed feedback.
\newblock {\em Discrete Contin. Dyn. Syst.}, 7:259--274, 2001.

\bibitem{Walther2001a}
H.-O. Walther.
\newblock Contracting returns maps for some delay differential equations.
\newblock In T.~Faria and P.~Freitas, editors, {\em Topics in Functional
  Differential and Difference Equations}, volume~29 of {\em Fields Institue
  Communications Series}, pages 349--360. AMS Providence, 2001.

\bibitem{Walther2003}
H.-O. Walther.
\newblock Stable periodic motion of a delayed spring.
\newblock {\em Topol. Methods Nonlinear Anal.}, 21:1--28, 2003.

\bibitem{Wu2003}
J.~Wu.
\newblock Stable phase-locked periodic solutions in a delay differential
  system.
\newblock {\em J.\ Differential Equations}, 194:237--286, 2003.

\bibitem{Xie1991}
X.~Xie.
\newblock Uniqueness and stability of slowly oscillating periodic solutions of
  delay equations with bounded nonlinearity.
\newblock {\em J. Dynam. Differential Equations}, 3:515--540, 1991.

\bibitem{Xie1992}
X.~Xie.
\newblock The multiplier equation and its applications to {$S$}-solutions of a
  differential delay equation.
\newblock {\em J. Differential Equations}, 95:259--280, 1992.

\bibitem{Xie1993}
X.~Xie.
\newblock Uniqueness and stability of slowly oscillating periodic solutions of
  differential delay equations with unbounded nonlinearity.
\newblock {\em J.\ Differential Equations}, 103:350--374, 1993.

\end{thebibliography}

\end{document}